\documentclass[10pt]{amsart}

\usepackage[top=1.5in,bottom=1in,left=1.1in,right=1.1in]{geometry}

\usepackage{amsmath}
\usepackage{amsthm}
\usepackage{graphicx}
\usepackage{amsfonts}
\usepackage{amssymb}
\usepackage{url}
\usepackage{amscd}

\def\bdrS{ \partial S}

\def\tM{\tilde M}
\def\Dt{\Delta}

\def\hS{\hat S}

\def\tu{\tilde u}

\def\zerS{Z}

\def\tanpr{T_{p_r}(\overline{U}_r)}

\def\cV{\mathcal V}

\def\tg{\tilde g}

\newcommand{\re}{\mathbb{R}}

\newcommand{\Z}{\mathbb{Z}}

\newcommand{\N}{\mathbb{N}}

\newcommand{\cv}[1]{\mbox{conv}(#1)}

\newcommand{\bdS}{\pt S}

\newcommand{\lmd}{\lambda}

\newcommand{\nn}{\nonumber}
\newcommand{\eps}{\epsilon}

\newcommand{\dt}{\delta}
\newcommand{\af}{\alpha}

\newcommand{\sig}{\sigma}

\newcommand{\supp}{\mbox{supp}}
\newcommand{\reff}[1]{(\ref{#1})}
\newcommand{\pt}{\partial}

\newcommand{\prm}{\prime}

\newcommand{\mc}[1]{\mathcal{#1}}

\newcommand{\bdes}{\begin{description}}
\newcommand{\edes}{\end{description}}

\newcommand{\bal}{\begin{align}}
\newcommand{\eal}{\end{align}}

\newcommand{\bnum}{\begin{enumerate}}
\newcommand{\enum}{\end{enumerate}}

\newcommand{\bit}{\begin{itemize}}
\newcommand{\eit}{\end{itemize}}

\newcommand{\bea}{\begin{eqnarray}}
\newcommand{\eea}{\end{eqnarray}}
\newcommand{\be}{\begin{equation}}
\newcommand{\ee}{\end{equation}}

\newcommand{\baray}{\begin{array}}
\newcommand{\earay}{\end{array}}

\newcommand{\bsry}{\begin{subarray}}
\newcommand{\esry}{\end{subarray}}

\newcommand{\bca}{\begin{cases}}
\newcommand{\eca}{\end{cases}}

\newcommand{\bcen}{\begin{center}}
\newcommand{\ecen}{\end{center}}

\newcommand{\bbm}{\begin{bmatrix}}
\newcommand{\ebm}{\end{bmatrix}}

\newcommand{\bmx}{\begin{matrix}}
\newcommand{\emx}{\end{matrix}}

\newcommand{\bpm}{\begin{pmatrix}}
\newcommand{\epm}{\end{pmatrix}}

\newcommand{\btab}{\begin{tabular}}
\newcommand{\etab}{\end{tabular}}

\theoremstyle{plain}
\newtheorem{thm}{Theorem}[section]
\newtheorem{theorem}[thm]{Theorem}
\newtheorem{pro}[thm]{Proposition}

\newtheorem{lemma}[thm]{Lemma}

\theoremstyle{definition}
\newtheorem{example}[thm]{Example}
\newtheorem{exm}[thm]{Example}

\renewcommand{\subsection}[1]{
    \stepcounter{subsection}
    \settowidth{\hangindent}{\bf\thesubsection.~}
    \hangafter=1
    \bigskip\noindent
    {\bf\hbox{\thesubsection.~}#1}\par
    \nobreak
    \medskip
}

\def\bbR{\mathbb R}
\def\tM{\tilde M}

\def\oU{\overline U}
\def\oM{\overline M}

\begin{document}

\title[Sufficient and Necessary Conditions for
Semidefinite Representability]
{Sufficient and Necessary Conditions for \\
 Semidefinite Representability of Convex Hulls and Sets}

%\title[Exact SDP Representation of Convex Hulls and Sets]
%{Exact Semidefinite Representation of Convex Hulls and Sets}

\author[J.~William Helton]{J.~William Helton}
\address{%William Helton:\vskip .1cm
Department of Mathematics, University of California at San Diego,
9500 Gilman Drive, La Jolla, CA 92093}
\email{helton@math.ucsd.edu}

\author{Jiawang Nie}
\address{%Jiawang Nie:\vskip .1cm
Department of Mathematics,  University of California at San Diego,
9500 Gilman Drive, La Jolla, CA 92093 }
\email{njw@math.ucsd.edu}


\begin{abstract}
A set $S\subseteq \re^n$ is called to be
{\it Semidefinite (SDP)} representable
if $S$ equals the projection of a set in higher dimensional space
which is describable by some Linear Matrix
Inequality (LMI).
Clearly, if $S$ is SDP representable,
then $S$ must be convex and semialgebraic
(it is describable by conjunctions and disjunctions of
polynomial equalities or inequalities).
This paper proves sufficient conditions
and necessary conditions for SDP representability
of convex sets and convex hulls
by proposing a new approach
to construct SDP representations.
%from those discussed in \cite{Las06,Par06,HN1},

The contributions of this paper are:
{\bf (i)} For bounded SDP representable sets $W_1,\cdots,W_m$,
we give an explicit construction of an SDP representation
for $\cv{\cup_{k=1}^mW_k}$.
This provides a technique for building global SDP representations
from the local ones.
{\bf (ii)}
For the SDP representability of a compact convex semialgebraic set $S$,
we prove sufficient:
%condition:
the boundary $\bdS$ is
nonsingular
and positively curved,
while necessary is:
$\bdS$ has nonnegative curvature at
each  nonsingular point.
%smooth points and on nondegenerate corners.
In terms of defining polynomials  for $S$,
nonsingular boundary amounts to them
%to defining polynomials  for $S$
having nonvanishing gradient at each point on $\bdS$
and the curvature condition can be expressed as
their strict versus nonstrict quasi-concavity of
%defining polynomials
at those points on $\bdS$ where they vanish.
%the strict versus nonstrict quasi-concavity of
%defining polynomials  on those points on $\bdS$
%where they vanish.
The gaps between them
are $\bdS$ having or not having singular points
either of the gradient or  of the curvature's
positivity.
%positive versus nonnegative curvature
%and smooth versus nonsmooth points.
A sufficient condition bypassing the gaps is when some
defining polynomials of $S$ satisfy an algebraic condition called
sos-concavity.
{\bf (iii)}
For the SDP representability of the
convex hull of a compact nonconvex semialgebraic set $T$,
we find that the critical object is $\pt_cT$,
the maximum subset of $\pt T$ contained in $\pt \cv{T}$.
We prove sufficient  for SDP representability:
$\pt_cT$ is
nonsingular and
positively curved,
and necessary is:
$\pt_cT$ has nonnegative curvature at
nonsingular points.
The gaps between our sufficient and necessary conditions
are similar to case (ii).
%For the SDP representability of
%convex hull of a compact nonconvex semialgebraic set $T$,
%we prove similar sufficient and necessary conditions
%in terms of curvature of its convex boundary $\pt_cT$
%(the maximum subset of $\pt T$ contained in $\pt \cv{T}$),
%i.e., the (strict) quasi-concavity of defining polynomials
%on the part of $\pt_cT$ where they vanish.
The positive definite Lagrange Hessian (PDLH)
condition,
which meshes well with constructions,
is also discussed.
\end{abstract}

\maketitle

\section{Introduction}

Semidefinite programming (SDP) \cite{BTN, NN94, N06, SDPhandbook}
is one of the main advances in convex optimization theory and applications.
It has a profound effect on combinatorial optimization, control theory
and nonconvex optimization as well as many other disciplines.
There are effective numerical algorithms
for solving problems presented in terms of Linear Matrix
Inequalities (LMIs).
One fundamental problem in semidefinite programming
and linear matrix inequality theory is
what sets can be presented in semidefinite programming.
This paper addresses one of the most classical aspects of this problem.

A set $S$ is said to have an {\it LMI representation} or be {\it
LMI representable} if
\be \nn
S=\{x\in \re^n: A_0+A_1 x_1 + \cdots + A_n x_n \succeq 0\}
\ee
for some symmetric matrices $A_i$. Here the notation $X
\succeq 0 \,(\succ 0)$ means the matrix $X$ is positive
semidefinite (definite). If $S$ has an interior point, $A_0$ can
be assumed to be positive definite without loss of generality.
Obvious necessary conditions for $S$ to be LMI representable are
that $S$ must be convex and
{\it basic closed semialgebraic}, i.e.,
\begin{align*}
S &=\{x\in \re^n:\, g_1(x)\geq 0, \cdots, g_m(x)\geq 0\}
\end{align*}
where $g_i(x)$ are multivariate polynomials.
%%and $\intr{S}$ denotes the interior of $S$.
It is known that not every convex
basic closed semialgebraic set
can be represented by LMI
(e.g., the set $\{x\in\re^2: x_1^4+x_2^4\leq 1\}$
is not LMI representable \cite{HV}).
If the convex set $S$ can be represented as
the projection to $\re^n$ of
\be
\label{eq:liftLMI}
\hS = \left\{ (x, u) \in \re^{(n+N)}: A_0 +
\sum_{i=1}^n A_i x_i + \sum_{j=1}^N B_j u_j \succeq 0 \right\}
\subset \re^{(n+N)},
\ee
that is
$S = \left\{ x \in \re^{n}: \exists u \in \re^n, \
(x,u) \in \hS \right\}
$,
for some symmetric matrices $A_i$ and $B_j$,
then $S$ is called {\it semidefinite
representable} or {\it SDP representable}.
Sometimes we refer to a
semidefinite representation as a {\it lifted LMI representation}
of the convex set $S$ and to the LMI in \reff{eq:liftLMI} as a
{\it  lifted LMI for $S$}, and to $\hS$
as the {\it SDP lift of $S$}.

\medskip

If $S$ has an SDP representation instead of LMI representation,
then $S$ might not be basic closed semialgebraic,
but it must be semialgebraic, i.e.,
$S$ is describable by conjunctions or disjunctions
of polynomial equalities or inequalities \cite{BCR}.
Furthermore,
%if $S$ is open,
%then $S$ is a union of {\it basic open semialgebraic sets}
the interior $\overset{\circ}{S}$ of $S$
is a union of {\it basic open semialgebraic sets}
(Theorem~2.7.2 in \cite{BCR}), i.e.,
$ \overset{\circ}{S} = \bigcup_{k=1}^m T_k$ for sets of the form
\[
T_k = \{ x\in \re^n:\, g_{j_1}(x) > 0,\cdots, g_{j_{m_{k}}}(x) > 0 \}.
\]
Here $g_{i_j}$ are all multivariate polynomials.
For instance, the set
\[
\left\{x\in\re^2:\, \exists \, u\geq 0, \,
\bbm x_2 & x_1-u \\ x_1-u & 1 \ebm \succeq 0 \right\}
\]
is not a basic semialgebraic set.
When $S$ is SDP representable, $S$ might not be closed,
but its closure $\bar{S}$ is a union of basic closed semialgebraic sets
(Proposition~2.2.2 and Theorem~2.7.2 in \cite{BCR}).
For example, the set
\[
\left\{x\in\re:\, \exists \, u, \,  \bbm x & 1 \\ 1 & u \ebm \succeq 0\right\}
= \{x\in\re:\, x> 0\}
\]
is not closed, but its closure is a basic closed semialgebraic set.
The content of this paper is to give  sufficient conditions
and (nearby) necessary conditions
for SDP representability of convex semialgebraic sets
or convex hulls of nonconvex semialgebraic sets.

\medskip

\noindent
{\bf History} \ \
Nesterov and Nemirovski
(\cite{NN94}), Ben-Tal and Nemirovski (\cite{BTN}), and Nemirovsky (\cite{N06})
gave collections of examples of SDP representable sets.
Thereby leading to the fundamental question which sets are SDP representable?
In \S 4.3.1 of his excellent ICM 2006 survey \cite{N06} Nemirovsky wrote ``
this question seems to be completely open". Obviously, to be SDP
representable, $S$ must be convex and semialgebraic.
What are the sufficient conditions that $S$ is SDP representable?
This is the main subject of this paper.

When $S$ is a basic closed semialgebraic set of the form
$\{x\in \re^n:\, g_1(x)\geq 0, \cdots, g_m(x)\geq 0\}$,
there is recent work on the SDP representability of $S$
and its convex hull.
Parrilo \cite{Par06} and Lasserre \cite{Las06} independently
proposed a natural
construction of lifted LMIs using moments and sum of squares techniques
with the aim  of producing  SDP representations.
Parrilo \cite{Par06} proved the construction gives
an SDP representation in the two dimensional case when the boundary
of $S$ is a single rational planar curve of genus zero.
Lasserre \cite{Las06} showed the construction can give arbitrarily accurate
approximations to compact $S$,
and the constructed LMI is a lifted LMI for $S$
by assuming {\it almost all} positive affine functions on $S$
have SOS representations with uniformly bounded degree.
Helton and Nie \cite{HN1} proved that this type of construction
for compact convex sets $S$ gives the exact SDP representation
under various hypotheses
on the Hessians of the defining polynomials $g_i(x)$,
and also gave other sufficient conditions for $S$ to be SDP representable.
Precise statements of most of the main theorems
in \cite{HN1} can be seen here in this paper in later sections
where they are used in our proofs,
see Theorems \ref{thm:qscv},
\ref{thm:setSimple} and
\ref{thm:posCurv}.

\iffalse
Under the assumption that
$S$ is compact convex and has nonempty interior,
Helton and Nie \cite{HN1} proved the following SDP representation results:
(i) If $g_i(x)$ are concave on $S$,
and the {\it positive definite Lagrange Hessian (PDLH)} condition
holds (for any nonzero vector
$\ell\in\re^n$, the Hessian of the Lagrange function
corresponding to the optimization problem of minimizing $\ell^Tx$
over $S$ is positive definite at the minimizer, i.e.,
$-\sum_{i=1}^m \lmd_i \nabla^2 g_i(u)$ is positive definite, where
$u$ is the minimizer and $\lmd_i\geq 0$ are Lagrange multipliers),
then $S$ is SDP representable;
(ii)
If each $g_i(x)$ is either sos-concave
( $-\nabla^2 g_i(x)= W(x)^TW(x)$ for some possibly
nonsquare matrix polynomial)
or strictly quasi-concave on $S$
(the super level sets $\{x\in S:\, g_i(x)\geq \af\}$ is convex
and its boundary has positive curvature),
then $S$ is SDP representable;
(iii) If each $g_i$ is either sos-concave or
$S_i =\{x\in \re^n:\, g_i(x)\geq 0\}$ is {\it poscurv-convex}
($S_i$ is compact convex and $\zerS_i=\{x\in\re^n:\, g_i(x)
= 0\}$ has positive curvature),
or $S_i$ is {\it
extendable poscurv-convex with respect to} $S$
($\exists$ a poscurv-convex set $T_i \supseteq S$ such that $\pt T_i\cap S =
\zerS_i \cap S$),
then $S$ is SDP representable.
\fi

\medskip

\noindent
{\bf Contributions} \,\,
In this paper, we prove sufficient and
(nearby)
 necessary conditions
for the SDP representability of convex sets and
convex hulls of nonconvex sets.
To obtain these conditions
%are obtained using a new
 we give a new and different construction of SDP representations,
which we combine with
those discussed in \cite{HN1,Las06,Par06}.
The following are our main contributions.

\medskip

{\it First}, consider the SDP representability
of the convex hull of union of sets $W_1,\cdots,W_m$
which are all SDP representable.
%The convex hull $\cv{\cup_{k=1}^m W_k}$
%is a convex semialgebraic set.
%Is it also SDP representable?
%This problem will be discussed in Section~\ref{sec:sdpunion}.
When every $W_k$ is bounded,
we give an explicit SDP representation
of $\cv{\cup_{k=1}^m W_k}$.
When some $W_k$ is unbounded,
we show that the closure of
the projection of the constructed SDP lift
is exactly the closure of $\cv{\cup_{k=1}^m W_k}$.
This is Theorem \ref{thm:mUsdp}.
It provides a new approach for constructing global SDP representations
from local ones,
and plays a key
role in proving our main theorems in Sections
\S \ref{sec:cvbsa} and \S \ref{sec:ncvx}.
%Theorem~\ref{thm:bcsa} and Theorem~\ref{thm:nonbcsa}.

\medskip

{\it Second}, consider the SDP representability of
a compact convex semialgebraic  set $S=\cup_{k=1}^m T_k$.
Here
$T_k=\{ x\in \re^n:\, g^k_1(x)\geq 0,\, \cdots,\, g^k_m(x)\geq 0 \}$
are defined by polynomials $g^k_i$;
note  each $T_k$ here is not necessarily convex.
%Assume each $T_k$ has interior near the boundary $\bdS$.
Denote by $Z(g)$ the zero set of a polynomial $g$.
Our main result for everywhere nonsingular boundary $\pt S$ is approximately:

\smallskip
\noindent
{\it Assume each $T_k$ has interior near $\pt T_k \cap \bdS$
and its boundary is nonsingular
(the defining polynomials  $g^k_i$
%are nonsingular
at every point $u\in \pt S \cap Z(g^k_i)$
satisfy $\nabla g^k_i(u) \not =0$).
Then sufficient for
$S$ to be SDP representable is:
every $\bdS\cap Z(g^k_i)$
has positive curvature
(i.e. $g^k_i$ is strictly quasi-concave on $\pt S \cap Z(g^k_i)$),
and necessary is: every irredundant $Z(g^k_i)$
has nonnegative curvature on $\bdS$
(i.e. $g^k_i$ is quasi-concave at $u$
whenever $Z(g^k_i)$ is irredundant at $u \in \pt S \cap Z(g^k_i)$).
}

\smallskip
\noindent
The notion of positive curvature we use is the standard one
of differential geometry, the notion of
quasi-concave function is the usual one and all of this
will be defined formally in \S 3.
To have
necessary conditions on a family $F$ of  defining functions for  $S$
we need an assumption that $F$ contains no functions irrelevant to
the defining of $S$.
Our notion of irredundancy plays a refinement
of this role.
The gaps
between our sufficient and necessary conditions
are $\bdS$ having positive versus nonnegative curvature
and singular  versus nonsingular points.
A case bypassing the gaps is that $g^k_i$ is sos-concave, i.e.,
$-\nabla^2 g^k_i(x)= W(x)^TW(x)$ for some possibly
nonsquare matrix polynomial $W(x)$.
Also when $\pt S$ contains singular points
$u$ we have  additional  conditions
which are sufficient:
for example, adding $-\nabla^2 g^k_i(u) \succ 0$
where $ \nabla g^k_i(u)  = 0$
to the hypotheses of the statement above
guarantees SDP representability.
We emphasize that our
conditions here concern only the
quasi-concavity properties of defining polynomials $g^k_i$
on the boundary $\bdS$ instead of on the whole set $S$.
See Theorems~\ref{thm:bcsa}, \ref{thm:cvnbsa},  \ref{thm:neccbsa}
and \ref{thm:intrpoly}
 for details.

\medskip

{\it Third}, consider the SDP representability
of the convex hull of a  compact nonconvex set $T=\cup_{k=1}^m T_k$.
Here $T_k = \{ x\in \re^n:\, f^k_1(x) \geq 0,\,
\cdots,\, f^k_{m_k}(x) \geq 0\}$ are defined by polynomials $f^k_i(x)$.
To obtain sufficient and necessary conditions,
we find that the critical object is the {\it convex boundary} $\pt_cT$,
the maximum subset of $\pt T$ contained in $\pt \cv{T}$.
Our main result for $\pt_c T$ having
%smooth points and nondegenerate corners
everywhere nonsingular boundary is approximately:

\smallskip

\noindent
{\it
Assume each $T_k$ has nonempty interior near $\pt_cT$
and the defining polynomials $f^k_i$
are nonsingular at every point $u\in \pt_cT \cap Z(f^k_i)$
(i.e. $\nabla f^k_i(u) \not =0$).
Then sufficient for
$\cv{T}$ to be SDP representable is:
every $\pt_cT \cap Z(f^k_i)$
has positive curvature
(i.e. $f^k_i$ is strictly quasi-concave on $\pt_cT \cap Z(f^k_i)$),
and necessary is: every irredundant $Z(f^k_i)$
has nonnegative curvature on $\bdS$
(i.e. $f^k_i$ is quasi-concave at $u$
whenever $f^k_i$ is irredundant at $u \in \pt_cT \cap Z(f^k_i)$).
}

\smallskip
\noindent
This generalizes our second  result (above) concerning
SDP representability of compact convex semialgebraic sets.
Also (just as before) we successfully weaken the hypothesis
in several directions, which covers various cases of singularity.
For example,
one other sufficient condition allows $f^k_i$ to be sos-concave.
When $T_k$ has empty interior,
we prove that a  condition called the
{\it positive definite Lagrange Hessian (PDLH) condition}
is sufficient.
See Theorems~\ref{thm:cvhbcsasdp}, \ref{thm:nonbcsa}, \ref{thm:ncvxnecc},
\ref{thm:pdlh} and \ref{thm:pdlhU}
for details.

\medskip

Let us comment on the constructions of lifted LMIs.
In this paper we analyze two different types of constructions.
One is a fundamental moment type relaxation due to Lasserre-Parrilo
which builds LMIs (discussed in \S \ref{sec:pdlh}),
while the other is a localization technique introduced in this paper.
The second result stated above is proved in two different ways,
one of which gives a refined result:\\
\smallskip
{\it
Given a basic closed semialgebraic set
$S=\ closure \; of \;
\{ x\in \re^n:\, g_1(x)> 0,\, \cdots,\, g_m(x) > 0 \}$
with nonempty interior.
If $S$ is convex and its
 boundary $\pt S$ is positively curved and nonsingular,
 then there exists a certain set of defining polynomials for $S$
for which a Lasserre-Parrilo type moment relaxation
gives the lifted LMI for $S$. \\
}
\smallskip
See \S \ref{sec:ghomi} for the proof.
A very different construction of lifted LMI
is also given in \S \ref{sec:ncvx}
using the localization technique plus a Lasserre-Parrilo type moment
construction.

%
%----------
%
%OLD
%
%NIE:\\
%Lastly, we show that the Lasserre-Parrilo type moment relaxation
%technique (see \cite{HN1}) for constructing lifted LMIs
%works under similar conditions like positive curvature.
%Given a basic closed semialgebraic set
%$S=\{ x\in \re^n:\, g_1(x)\geq 0,\, \cdots,\, g_m(x)\geq 0 \}$
%that is convex and has nonempty interior,
%we prove that the Lasserre-Parrilo type moment relaxation
%gives the exact lifted LMI for $S$
%by finding new defining polynomials for $S$,
%under the condition that
%the boundary $\pt S$ is positively curved and nonsingular.
%The basic technique for finding new defining polynomials for $S$
%is based on Ghomi \cite{Ghomi} and Helton and Nie \cite{HN1}.
%See the Appendix Section~\ref{sec:ghomi} for more details.
%
%end NEW \\
%----------------------------------------------------------------------

\medskip

\noindent
{\bf Notations and Outline}
The following notations will be used.
A polynomial $p(x)$ is said to be a {\it sum of squares (SOS)} if
$p(x) = w(x)^Tw(x)$ for some column vector polynomial $w(x)$.
A matrix polynomial $H(x)$ is said to be SOS if
$H(x) = W(x)^TW(x)$ for some possibly nonsquare matrix polynomial $W(x)$.
\iffalse
%The necessary condition for $p(x)$ to be SOS is that it is nonnegative
%on the whole space $\re^n$, but the converse might not be true. We
%refer to \cite{Rez00} for a survey on SOS polynomials.
%A symmetric
%matrix polynomial $P(x) \in \re[x]^{n\times n}$ is SOS if there
%exists a matrix polynomial $W(x)$ with $n$ columns such that $P(x)
%=W(x)^TW(x)$.
%A polynomial $p(x)$ is called {\it
%sos-concave} if its negative Hessian
%$-\nabla^2 p(x)=
%-\left(\frac{\pt^2 p}{\pt x_k \pt x_\ell}\right)$
%is SOS.
%Similarly, $p(x)$ is called {\it sos-convex} if its Hessian
%$\nabla^2 p(x)$ is SOS.
%The set $\{x\in\re^n: p(x) \geq 0\}$ is
%called {\it sos-convex} if the polynomial $p(x)$ is sos-concave.
\fi
$\N$ denotes the set of nonnegative integers,
$\re^n$ denotes the Euclidean space of $n$-dimensional space
of real numbers,
$\re_+^n$ denotes the nonnegative orthant of $\re^n$.
$\Dt_m=\{\lmd\in\re_+^m: \lmd_1+\cdots+\lmd_m=1\}$
is the standard simplex.
For $x\in\re^n$, $\|x\|=\sqrt{\sum_{i=1}^n x_i^2}$.
$B(u,r)$
denotes the open ball
$\{x\in\re^n:\, \|x-u\| < r\}$
and $\bar{B}(u,r)$ denotes the closed ball
$\{x\in\re^n:\, \|x-u\| \leq r\}$.
For a given set $W$, $\overline{W}$ denotes the closure of $W$,
and $\pt W$ denotes its topological boundary.
For a given matrix $A$, $A^T$ denotes its transpose.
$I_n$ denotes the $n\times n$ identity matrix.
\iffalse
%For $\af \in \N^n$, $|\af| := \af_1+\cdots+\af_n$,
%$x^\af := x_1^{\af_1}\cdots x_n^{\af_n}$.
%Given a set $K\subset \re^n$,
%Given $\af\in\N^n$ and $f(x)\in C^k(K)$,
%$D^\af f(x) := \frac{\pt^{|\af|} f(x)}
%{\pt x_1^{\af_1}\cdots \pt x_n^{\af_n}}$.
%For a symmetric matrix $A$,
%$\lmd_{\min}(A)$ denotes the smallest eigenvalue
%of $A$.
\fi

\medskip
The paper is organized as follows.
Section~\ref{sec:sdpunion} discusses the SDP representation of the convex hull of
union of SDP representable sets.
Section~\ref{sec:cvbsa} discusses the SDP representability
of convex semialgebraic sets.
Section~\ref{sec:ncvx} discusses the SDP representability
of convex hulls of nonconvex semialgebraic sets.
%Section~\ref{sec:clc} draws some conclusions.
Section~\ref{sec:ghomi} %is an appendix which
presents a similar version of Theorem \ref{thm:bcsa}
and gives a different but more geometric proof
based on results of \cite{HN1}.
Section~\ref{sec:concl} concludes this paper
and makes a conjecture.
%which is convex geoemetrical and is very different
%from the one in \S \ref{sec:cvbsa}.
%

\bigskip
\section{The convex hull of union of SDP representable sets}
\setcounter{equation}{0}
\label{sec:sdpunion}

It is obvious the intersection of SDP representable
sets is also SDP representable,
but the union might not be because
the union may not be convex.
However, the convex hull of the union of
SDP representable sets is a convex semialgebraic set.
Is it also SDP representable?
This section will address this issue.

Let $W_1,\cdots, W_m \subset \re^n$ be convex sets.
Then their Minkowski sum
\[
W_1+\cdots+W_m =
\left\{x = x_1+\cdots + x_m: \, x_1 \in W_1, \cdots, x_m \in W_m \right\}
\]
is also a convex set.
If every $W_k$ is given by some lifted LMI,
then a lifted LMI for $W_1+\cdots+W_m$
can also be obtained immediately by definition.
Usually the union of convex sets $W_1,\cdots, W_m$ is no longer convex,
but its convex hull $\mbox{conv}(\cup_{k=1}^m W_k)$ is convex again.
Is $\mbox{conv}(\cup_{k=1}^m W_k)$ SDP representable if
every $W_k$ is?
We give a  lemma first.
\iffalse
%Let $W_1,\cdots, W_m$ be convex sets given by SDP representations
%\begin{align*}
%W_k  = \left\{ x \in \re^n:\,\exists\,\, u^{(k)}, \,\,
%A^{(k)} + \sum_{i=1}^n x_i B_i^{(k)}
% + \sum_{j=1}^{N_k} u_j^{(k)} C_j^{(k)} \succeq 0\right\}
%\end{align*}
%for symmetric matrices $A^{(k)},B_i^{(k)}, C_j^{(k)}$.
%Clearly, the Minkowski sum $W_1+\cdots+W_m$
%is a convex set with SDP representation
%{\small
%\[
%\left\{ \sum_{k=1}^m x^{(k)}:\,\exists\,u^{(k)},\,
%A^{(k)} + \sum_{i=1}^n x_i^{(k)} B_i^{(k)}
% + \sum_{j=1}^{N_k} u_j^{(k)} C_j^{(k)} \succeq 0,\, k=1,\cdots,m\right\}.
%\]
%}
%\noindent
%We are interested in finding the SDP representation
%of the convex hull of
%the union $\cup_{k=1}^m W_k$.
\fi
\begin{lemma}
\label{lem:cvhdef}
If $W_k$ are all nonempty convex sets, then
\begin{align*}
\mbox{conv}(\bigcup_{k=1}^m W_k) =
\bigcup_{ \lmd \in \Dt_m }
\left(\lmd_1 W_1 + \cdots + \lmd_m W_m \right)
\end{align*}
where $\Dt_m = \{ \lmd\in \re_+^m:\, \lmd_1+\cdots+\lmd_m =1 \}$
is the standard simplex.
\end{lemma}
\begin{proof}
This is a special case of Theorem~3.3 in Rockafellar \cite{Rockf}.
\iffalse
%The right hand side is obviously contained in the left hand side.
%It suffices to prove
%\begin{align*}
%\cv{\bigcup_{k=1}^m W_k} \subseteq
%\bigcup_{ \lmd \in \Dt_m }
%\left(\lmd_1 W_1 + \cdots + \lmd_m W_m \right).
%\end{align*}
%For any $x\in \cv{\cup_{k=1}^m W_k}$,
%there exist $\mu^{(k,i)} \geq 0, y^{(k,i)}\in W_k$ such that
%\[
%x = \sum_{k=1}^m \sum_{i=1}^{L_k} \mu^{(k,i)} y^{(k,i)}, \quad
%%+ \cdots +
%%\sum_{i=1}^{L_m} \mu_i^{(m,i)} y^{(m,i)}
%\sum_{k=1}^m \sum_{i=1}^{L_k} \mu^{(k,i)} =1.
%\]
%Now let
%$\mu_k = \sum_{i=1}^{L_k} \mu^{(k,i)}$ and
%\[
%x^{(k)} =
%\bca \frac{1}{\mu_k} \sum_{i=1}^{L_k} \mu^{(k,i)} y^{(k,i)} & \text{ if } \mu_k >0 \\
%w^{(k)} & \text{ if } \mu_k =0
%\eca
%\]
%where $w^{(k)}$ is any point from $W_k$.
%Since $W_k$ is convex, we know $x^{(k)}\in W_k$.
%Obviously $\mu_k\geq 0$ and $\mu_1+\cdots+\mu_m=1$.
%So
%\[
%x = \mu_1 x^{(1)} + \cdots + \mu_m x^{(m)} \in
%\bigcup_{ \lmd \in \Dt_m }
%\left(\lmd_1 W_1 + \cdots + \lmd_m W_m \right)
%\]
%which completes the proof.
\fi
\end{proof}

%
%\bigskip
%
%For each $\lmd \in \Dt_m$ and $0<\lmd <1$,
%the set
%$\lmd_1 W_1 + \cdots + \lmd_m W_m $ is convex
%and has the representation
%\begin{smaller}
%\begin{align*}
%\left\{ \sum_{k=1}^m x^{(k)}:\,\exists\,u^{(k)},
%\ \lmd_k A^{(k)} + \sum_{i=1}^n x_i^{(k)} B_i^{(k)}
% + \sum_{j=1}^{N_k} u_j^{(k)} C_j^{(k)} \succeq 0,\, 1\leq k\leq m \right\}.
%\end{align*}
%\end{smaller}
%Define $\mc{C}(W_1,\cdots,W_m) =
%\left\{ x : \,\exists\, \lmd\in\Dt_m,\, x \in  \lmd_1 W_1 + \cdots + \lmd_m W_m\right\}$
%which is the same as
%\begin{smaller}
%\begin{align} \label{eq:cvhlmi}
%\left\{ \sum_{k=1}^m x^{(k)}: \exists\, \lmd\in\Dt_m,\,\exists\,u^{(k)},\,
% \lmd_k A^{(k)} + \sum_{i=1}^n x_i^{(k)} B_i^{(k)}
% + \sum_{j=1}^{N_k} u_j^{(k)} C_j^{(k)} \succeq 0,\, 1\leq k\leq m \right\}.
%\end{align}
%\end{smaller}
%

\medskip

Based on Lemma~\ref{lem:cvhdef},
given SDP representable sets $W_1,\cdots, W_m$,
it is possible to obtain a SDP representation for
the convex hull $\mbox{conv}(\cup_{k=1}^m W_k)$
directly from the lifted LMIs of all $W_k$
under rather weak conditions.
This is summarized in the following theorem.

\begin{theorem} \label{thm:mUsdp}
Let $W_1,\cdots, W_m$ be nonempty convex sets given by SDP representations
\begin{align*}
W_k  = \left\{ x \in \re^n:\,\exists\,\, u^{(k)}, \,\,
A^{(k)} + \sum_{i=1}^n x_i B_i^{(k)}
 + \sum_{j=1}^{N_k} u_j^{(k)} C_j^{(k)} \succeq 0\right\}
\end{align*}
for some symmetric matrices $A^{(k)},B_i^{(k)}, C_j^{(k)}$.
Define a new set
\begin{align} \label{eq:cvhlmi}
\mc{C} = \left\{ \sum_{k=1}^m x^{(k)}: \exists\, \lmd\in\Dt_m,\,\exists\,u^{(k)},\,
 \lmd_k A^{(k)} + \sum_{i=1}^n x_i^{(k)} B_i^{(k)}
 + \sum_{j=1}^{N_k} u_j^{(k)} C_j^{(k)} \succeq 0,\, 1\leq k\leq m \right\}.
\end{align}
Then we have the inclusion
\begin{align} \label{eq:u<=}
\mbox{conv}( \bigcup_{k=1}^m  W_k )\subseteq  \mc{C}
\end{align}
and the equality
\begin{align} \label{eq:cl=}
\overline{\mc{C}} =
\overline{\mbox{conv}( \bigcup_{k=1}^m  W_k )}.
\end{align}
In addition, if every $W_k$ is bounded, then
\begin{align} \label{eq:u=}
\mc{C} = \mbox{conv}( \bigcup_{k=1}^m  W_k ).
\end{align}
\end{theorem}

\noindent
{\it Remark:}
When some $W_k$ is unbounded,
$\mc{C}$ and
$\cv{\cup_{k=1}^m  W_k}$ might not be equal,
but they have the same interior,
which is good enough for solving optimization problems
over $\cv{\cup_{k=1}^m  W_k}$.

\begin{proof}
First, by definition of $\mc{C}$,
\reff{eq:u<=} is implied immediately
by Lemma~ \ref{lem:cvhdef}.
%First, show $\mbox{conv}(W_1 \cup W_2 )\subseteq  \mc{C}(W_1, W_2)$.
%For any $x \in \cv{W_1 \cup W_2}$, we have
%$x\in \lmd W_1 + (1-\lmd) W_2$ for some $\lmd \in [0,1]$,
%which clearly belongs to $\mc{C}(W_1, W_2)$.

\medskip

Second, we prove \reff{eq:cl=}.
By \reff{eq:u<=}, it is sufficient to prove
\[
\mc{C} \subseteq \overline{ \cv{ \bigcup_{k=1}^m W_k } }.
\]
Let $x = x^{(1)}+\cdots+x^{(m)}\in\mc{C}$;
then there exist $\lmd\in \Dt_m$ and $u^{(k)}$ such that
\be
\label{eq:xinSumW}
\lmd_k A^{(k)} + \sum_{i=1}^n x_i^{(k)} B_i^{(k)}
+ \sum_{j=1}^{N_k} u_j^{(k)} C_j^{(k)} \succeq 0,\,\,\,\, 1\leq k\leq m.
\ee
Without loss of generality, assume
$\lmd_1=\cdots=\lmd_\ell=0$
and $\lmd_{\ell+1},\cdots,\lmd_m>0$.
Then for $k=\ell+1,\cdots,m$, we have $\frac{1}{\lmd_k}x^{(k)}\in W_k$ and
\[
x^{(\ell+1)}+\cdots+x^{(m)} =
\lmd_{\ell+1} \frac{1}{\lmd_{\ell+1}}x^{(\ell+1)} + \cdots +
\lmd_{m} \frac{1}{\lmd_{m}}x^{(m)} \in \cv{ \bigcup_{k=1}^m W_k }.
\]
Since $W_k \ne \emptyset$,
there exist $ y^{(k)} \in W_k$ and $v^{(k)}$ such that
\[
A^{(k)} + \sum_{i=1}^n y_i^{(k)} B_i^{(k)}
 + \sum_{j=1}^{N_k} v_j^{(k)} C_j^{(k)} \succeq 0.
\]
For this and \reff{eq:xinSumW},
for arbitrary $\eps >0$ small enough, we have
\begin{smaller}
\begin{smaller}
\begin{align} \label{eq:eps}
\eps A^{(k)} + \sum_{i=1}^n (x_i^{(k)}+\eps y_i^{(k)}) B_i^{(k)}
 + \sum_{j=1}^{N_k} (u_j^{(k)}+\eps v_j^{(k)}) C_j^{(k)} \succeq 0,
 \ \  \ when \ 1 \leq k \leq \ell
\end{align}
\begin{align} \label{eq:1-elleps}
\frac{1-\ell\eps}{1+(m-\ell)\eps} \left\{
(\lmd_k+\eps) A^{(k)} + \sum_{i=1}^n (x_i^{(k)}+\eps y_i^{(k)}) B_i^{(k)}
 + \sum_{j=1}^{N_k} (u_j^{(k)}+\eps v_j^{(k)}) C_j^{(k)} \right\}
 \succeq 0, \ \ \ \  \ell+1 \leq k \leq m.
\end{align}
\end{smaller}
\end{smaller}
Now we let
\[
x^{(k)}(\eps):= x^{(k)}+\eps y^{(k)} \ ( 1\leq k\leq \ell), \qquad
x^{(k)}(\eps):= \frac{1-\ell\eps}{1+(m-\ell)\eps}
( x^{(k)}+\eps y^{(k)} )\ (\ell+1\leq k \leq m),
\]
\[
\lmd_k(\eps) = \eps \,(1\leq k\leq \ell), \qquad
\lmd_k(\eps) = \frac{1-\ell\eps}{1+(m-\ell)\eps} (\lmd_k+\eps) \
(\ell+1\leq k\leq m).
\]
In
this notation  (\ref{eq:eps})(\ref{eq:1-elleps}) become
\be
\label{eq:fullelleps}
\lmd_k(\eps) A^{(k)} +
\sum_{i=1}^n
\ x_i^{(k)}(\eps)   B_i^{(k)}
 + \sum_{j=1}^{N_k} \tu_j^{(k)}(\eps)
   C_j^{(k)}
 \succeq 0,\ \ \  \ 1 \leq k \leq m
\ee
with $ \tu_j^{(k)}(\eps) := (u_j^{(k)}+\eps v_j^{(k)})$.
Obviously $\lmd(\eps) \in \Dt_m$ and $0<\lmd_k(\eps)<1$
for every $1\leq k\leq m$. From
LMI  \reff{eq:fullelleps} and from
$\lmd_k(\eps)>0$ we get $\frac{1}{\lmd_{k}(\eps)}x^{(k)}(\eps) \in W_k$
for all $k$.
%By definition of $\mc{C}$ in \reff{eq:cvhlmi},
%\reff{eq:eps} and \reff{eq:1-elleps} imply
%$x(\eps) \in \mc{C}$.
%
Let $x(\eps):=x^{(1)}(\eps)+\cdots+x^{(m)}(\eps)$;
then we have
\[
x(\eps)=
\lmd_{1}(\eps) \frac{1}{\lmd_{1}(\eps)}x^{(1)}(\eps) + \cdots +
\lmd_{m}(\eps) \frac{1}{\lmd_{m}(\eps)}x^{(m)}(\eps)
\in \cv{ \bigcup_{k=1}^m W_k }.
\]
As $\eps \to 0$, $x(\eps) \to x$, which implies
$
x \in \overline{ \cv{ \bigcup_{k=1}^m W_k }}
$.

\medskip
Third, we prove \reff{eq:u=}.
When every $W_k$ is bounded,
it suffices to show
$ \mc{C}  \subseteq \cv{ \cup_{k=1}^m W_k}$.
Suppose $x=x^{(1)}+\cdots+x^{(m)}\in \mc{C}$
with some $\lmd \in [0,1]$ and $u^{(k)}$.
Without loss of generality, assume
$\lmd_1=\cdots=\lmd_\ell=0$
and $\lmd_{\ell+1},\cdots,\lmd_m>0$.
Obviously, for every $k=\ell+1,\cdots,m$,
we have $\frac{1}{\lmd_k}x^{(k)}\in W_k$ and
\[
x^{(\ell+1)}+\cdots+x^{(m)} =
\lmd_{\ell+1} \frac{1}{\lmd_{\ell+1}}x^{(\ell+1)} + \cdots +
\lmd_{m} \frac{1}{\lmd_{m}}x^{(m)} \in \cv{ \bigcup_{k=1}^m W_k }.
\]
Since $W_k \ne \emptyset$,
there exist $ y^{(k)} $ and $v^{(k)}$ such that
\[
A^{(k)} + \sum_{i=1}^n y_i^{(k)} B_i^{(k)}
 + \sum_{j=1}^{N_k} v_j^{(k)} C_j^{(k)} \succeq 0.
\]
Combining
the above with \reff{eq:xinSumW} and
observing $\lmd_1=\cdots=\lmd_\ell=0$,
we obtain that
\[
A^{(k)} + \sum_{i=1}^n (y_i^{(k)}+\af x_i^{(k)}) B_i^{(k)}
+ \sum_{j=1}^{N_k} (v_j^{(k)}+\af u_j^{(k)}) C_j^{(k)}
\succeq 0,\ \  \ \forall \af >0, \ \, \forall \,\, \ 1\leq k\leq \ell.
\]
Hence, we must have $x^{(k)}=0$ for $k=1,\cdots,\ell$, because otherwise
\[
y^{(k)} + [0,\infty) x^{(k)}
\]
is an unbounded ray in $W_k$,
which contradicts the boundedness of $W_k$. Thus
\[
x = x^{(\ell+1)}+\cdots+x^{(m)} \in  \cv{ \bigcup_{k=1}^m W_k }
\]
which completes the proof.
\end{proof}

\begin{example}
When some $W_k$ is unbounded,
$\mc{C}$ and $\cv{\cup_{k=1}^m W_k}$ might not be equal,
and $\mc{C}$ might not be closed.
Let us see some examples.\\
(i) Consider
$
W_1 = \left\{x\in \re^2:\, \bbm x_1 & 1 \\ 1 & x_2 \ebm \succeq 0\right\},
W_2 = \{0\}.
$
The convex hull $\mbox{conv}(W_1 \cup W_2) =
\{x\in \re_+^2:\, x_1+x_2=0 \text{ or } x_1x_2 >0 \}$.
However,
\[
\mc{C} = \left\{x\in \re^2:\,\exists\,\ \ 0\leq \lmd_1 \leq 1,\,
\bbm x_1 & \lmd_1 \\ \lmd_1 & x_2 \ebm \succeq 0\right\} = \re_+^2.
\]
$\mc{C}$ and $\cv{W_1\cup W_2}$ are not equal.\\
(ii) Consider
$
W_1 = \left\{x\in \re^2:\, \bbm x_1 & 1+x_2 \\ 1+x_2 & 1+u \ebm \succeq 0\right\}
$
and $W_2=\{0\}$. We have
$\mbox{conv}(W_1 \cup W_2) =
\{x\in \re^2:\, x_1 > 0, \text{ or } x_1=0 \text{ and } -1 \leq x_2 \leq 0 \}$
and $\overline{\mbox{conv}(W_1 \cup W_2)}= \{x\in \re^2:\, x_1 \geq 0 \}$.
But $\mc{C} = \mbox{conv}(W_1 \cup W_2)$ is not closed.
\end{example}

\medskip

\begin{example} Now we see some examples showing that the boundedness of
$W_1,\cdots,W_m$ is not necessary for \reff{eq:u=} to hold. \\
%$\mc{C}$ and $\cv{\cup_{k=1}^mW_k}$ have the same closure,
%and they are equal when all $W_k$ are bounded.
%\noindent
%(a) Consider the special case that
%each $W_k$ consists of a single point,
%i.e., $W_k=\{x\in\re^n: x - a^{(k)}=0\}$.
%By Theorem~\ref{thm:mUsdp}, we have
%\[
%\cv{\bigcup_{k=1}^m W_k} =
%\{ x^{(1)}+\cdots+x^{(m)}:\, \exists\, \lmd \in \Dt_m,\,
%x^{(k)} - \lmd_k a^{(k)}=0, 1\leq k\leq m
%\}
%\]
%which is exactly the convex hull
%$\{ \lmd_1 a^{(1)}+\cdots+\lmd_m a^{(m)}:
%\lmd\in \re_+^m,\,
%\lmd_1+\cdots+\lmd_m  =1
%\}$. \\
\noindent
(a) Consider the special case that
each $W_k$ is homogeneous, i.e.,
i.e., $A^{(k)}=0$ in the SDP representation of $W_k$.
Then by Lemma~\ref{lem:cvhdef},
we immediately have
\[
\mc{C} = \cv{\bigcup_{k=1}^m W_k}.
\]
%\noindent
%(c) Consider
%\[
%W_1 = \left\{ x\in \re^2:
%\bbm -x_1 & x_2+1 \\ x_2+1 & 2+x_1  \ebm \succeq 0
%\right\},
%W_2 = \left\{ x\in \re^2:
%\bbm 2-x_1 & x_2-1 \\ x_2-1 & x_1  \ebm \succeq 0
%\right\}.
%\]
%The convex hull of $W_1\cup W_2$ is given by
%\[
%\left\{x+y\in \re^2:\, \exists\,\lmd\in\Dt_2,
%\bbm -x_1 & x_2+\lmd_1 \\ x_2+ \lmd_1 & 2\lmd_1+x_1  \ebm \succeq 0,
%\bbm 2\lmd_2-y_1 & y_2-\lmd_2 \\ y_2-\lmd_2 & y_1  \ebm \succeq 0
%\right\}.
%\]
%This is justified by Theorem~\ref{thm:mUsdp},
%since they are both bounded.
\noindent
(b)
Consider
$
W_1 = \left\{ x\in \re^2:
\bbm -x_1 & 1 \\ 1 & x_2  \ebm \succeq 0
\right\},
W_2 = \left\{ x\in \re^2:
\bbm x_1 & 1 \\ 1 & x_2  \ebm \succeq 0
\right\}.
$
It can be verified that $\cv{W_1\cup W_2}$ is given by
\[
\mc{C} =
\left\{x+y\in \re^2:\, \exists\,\lmd\in\Dt_2,
\bbm -x_1 & \lmd_1 \\ \lmd_1 & x_2  \ebm \succeq 0,
\bbm y_1 & \lmd_2 \\ \lmd_2 & y_2  \ebm \succeq 0
\right\}
=\{x\in\re^2:\, x_2>0\}.
\]
%
%\noindent
%(e) Consider
%\[
%W_1 = \left\{ x\in \re^2:
%\bbm -x_1 & x_2+1 \\ x_2+1 & 2+x_1  \ebm \succeq 0
%\right\},
%W_2 = \left\{ x\in \re^2:
%\bbm x_1 & 1 \\ 1 & x_2  \ebm \succeq 0
%\right\}.
%\]
%The closure of $\cv{W_1\cup W_2}$ is given by the closure of
%\[
%\left\{x+y\in \re^2:\,  \exists\,\lmd\in\Dt_2,
%\bbm -x_1 & x_2+\lmd_1 \\ x_2+ \lmd_1 & 2\lmd_1+x_1  \ebm \succeq 0,
%\bbm y_1 & \lmd_2 \\ \lmd_2 & y_2  \ebm \succeq 0
%\right\}.
%\]
\end{example}

\bigskip
\section{Sufficient and necessary conditions for SDP representable sets}
%%{SDP representation of convex semialgebraic sets}
%%{Positively curved semialgebraic sets are SDP representable}
\label{sec:cvbsa}
\setcounter{equation}{0}

In this section,
we present sufficient conditions and necessary conditions
for SDP representability of a compact convex semialgebraic set $S$.
As we will see,
these sufficient conditions and necessary conditions
are very close with the main gaps being
between the boundary $\pt S$ having positive versus nonnegative
curvature and between the defining polynomials being singular or not
on the part of the boundary where they vanish.
A case which bypasses the gaps is when
some defining polynomials are sos-concave,
i.e., their negative Hessians are SOS.

Our approach is to start with
convex sets which are basic semialgebraic,
and to give weaker sufficient conditions than those given in \cite{HN1}:
the defining polynomials
are either sos-concave or strictly quasi-concave
on the part of the boundary $\pt S$ where they vanish
(not necessarily on the whole set).
%This improves the results given in \cite{HN1}.
And then we give
similar sufficient conditions
for convex sets that are {\em not} basic semialgebraic.
Lastly, we give necessary conditions
for SDP representability:
the defining polynomials are quasi-concave
on nonsingular points
%smooth points or nondegenerate corners
on the part of the boundary of $S$ where
they vanish.

Let us begin with reviewing
some background about curvature and quasi-concavity.
The key technique for proving the sufficient conditions
is to localize
to small balls containing a piece of $\bdrS$,
use the strictly quasi-concave function results
(Theorem~2 in \cite{HN1}) to represent
these small sets,
and then to apply Theorem~\ref{thm:mUsdp}
to patch all of these representations together,
thereby obtaining an SDP representation of $S$.

%\subsection{Background from \cite{HN1}}
\subsection{Curvature and quasi-concavity}

We first review the definition of {\it curvature}.
For a smooth function $f(x)$ on $\re^n$,
suppose the
zero set $Z(f):=\{x\in\re^n:\, f(x)=0\}$
is nonsingular
at a point $u\in Z(f)$, i.e., $\nabla f(u)\ne 0$.
Then $Z(f)$ is a smooth hypersurface near the point $u$.
$Z(f)$ is said to have {\it positive curvature} at
the nonsingular point $u\in Z(f)$
if its {\it second fundamental form } is positive definite, i.e.,
\be \label{def:posCurv}
 -v^T\nabla^2 f(u)v > 0 , \  \ \forall\,
0\ne v \in \nabla f(u)^\perp
\ee
where $\nabla f(u)^\perp:= \{v\in\re^n:\,\nabla f(u)^Tv=0\}$.
For a subset $V\subset Z(f)$, we say
$Z(f)$ has {\it positive curvature on $V$} if
$f(x)$ is nonsingular on $V$ and
$Z(f)$ has positive curvature at every $u \in V$.
When $>$ is replaced by $\geq$ in \reff{def:posCurv},
we can similarly define $Z(f)$ has {\it nonnegative curvature} at $u$.
We emphasize that this definition applies to
any zero sets defined by smooth functions
on their nonsingular points.
This is needed in \S\ref{sec:ghomi}.
We refer to Spivak \cite{Spivak} for more on
curvature and the second fundamental form.

The sign ``$-$" in the front of \reff{def:posCurv} might look confusing
for some readers, since $Z(f)$ and $Z(-f)$ define exactly the same zero set.
Geometrically, the curvature of a hypersurface should be
independent of the sign of the defining functions.
The reason for including the minus sign in \reff{def:posCurv}
is we are interested in the case where the set $\{x: f(x)\geq 0\}$
is locally convex near $u$ when $Z(f)$ has positive curvature at $u$.
%To make it more clear about the notion of positive curvature,
Now we give more geometric perspective by describing
alternative formulations of positive curvature.
Geometrically, the zero set $Z(f)$ has
nonnegative (resp. positive) curvature
at a nonsingular point $u \in Z(f)$
if and only if there exists an open set $\mc{O}_u$
such that $Z(f)\cap \mc{O}_u$ can be represented as the graph of
a function $\phi$ which is (strictly) convex at the origin in an
{\it appropriate} coordinate system
(see Ghomi \cite{Ghomi}).
Here we define a function to be convex (resp. strictly convex)
at some point if its Hessian is positive semidefinite
(resp. definite) at that point.
Also note when $Z(f)$ has positive curvature at $u$,
the set $\{x: f(x)\geq 0\}$ is locally convex near $u$
if and only if \reff{def:posCurv} holds,
or equivalently the set $\{x: f(x) \geq 0\} \cap \mc{O}_u$ is above the graph of function $\phi$.
%that strictly convex function under that coordinate system.
Now we prove the statements above and show such $\phi$ exists.
When the gradient $\nabla f(u) \ne 0$,
by the Implicit Function Theorem,
in an open set near $u$ the hypersurface $Z(f)$
can be represented as the graph of some smooth function
in a certain coordinate system.
Suppose the origin of this coordinate system corresponds to the point $u$,
and the set $\{x: f(x)\geq 0\}$ is locally convex near $u$.
%Then the function $\phi$ in these coordinates
%is convex (resp. strictly convex) at the origin if and only if
%$Z(f)$ has nonnegative (resp. positive) curvature at $u$.
%To be more specific,
Let us make the affine linear
coordinate transformation
\be \label{newcoord}
x-u= \bbm \nabla f(u) &  G(u) \ebm^T \bbm y \\ x^\prm \ebm
\ee
where $(y,x^\prm)\in \re \times \re^{n-1}$ are new coordinates
and $G(u)$ is an orthogonal basis for subspace $\nabla f(u)^\perp$.
By the Implicit Function Theorem, since  $\nabla f(u) \ne 0$,
in some neighborhood $\mc{O}_u$ of $u$,
the equation $f(x)=0$ defines
a smooth function $y=\phi(x^\prm)$.
For simplicity, we reuse the letter $f$ and
write $f(x^\prm, y)= f(x^\prm, \phi(x^\prm))=0$.
Since $\nabla f(u)$ is orthogonal to $G(u)$,
we have $f_y(0,0) =1$ and $\nabla_{x^\prm} \phi(0) =0$.
Twice differentiating $f(x^\prm,y)=0$ gives
\[
\nabla_{x^\prm x^\prm} f   +  \nabla_{x^\prm y} f  \nabla_{x^\prm} \phi^T
+ \nabla_{x^\prm} \phi  \nabla_{x^\prm y} f^T + f_{yy} \nabla_{x^\prm} \phi  \nabla_{x^\prm} \phi^T
+ f_y \nabla_{x^\prm x^\prm} \phi = 0.
\]
Evaluate the above at the origin in the new coordinates $(y,x^\prm)$,
to get
\[
\nabla_{x^\prm x^\prm} \phi (0) = - \nabla_{x^\prm x^\prm} f(u).
\]
So we can see $Z(f)$ has positive (resp. nonnegative) curvature at $u$
if and if the function $y = \phi(x^\prm)$ is strictly convex
(resp. convex) at $u$.
Since at $u$ the direction $\nabla f(u)$ points to the inside of the set $\{x:f(x)\geq 0\}$,
the intersection $\{x: f(x)\geq 0\} \cap \mc{O}_u$ lies above
the graph of $\phi$.
%Summarizing the above,
%when the set $\{x: f(x)\geq 0\}$ is locally convex near $u\in Z(f)$,
%the hypersurface $Z(f)$ has positive curvature at $u$
%if and only if \reff{def:posCurv} holds,
%or equivalently, there exists an open neighborhood  $\mc{O}_u$ of $u$
%such that $Z(f) \cap \mc{O}_u$ is the graph of a strictly convex function
%and $\{x: f(x)\geq 0\} \cap \mc{O}_u$ lies above its graph
%under a proper chosen coordinate system.

The notion of positive curvature of a nonsingular hypersurface
$Z(f)$ does not distinguish one side of $Z(f)$ from the other.
For example, the boundary of
the unit ball $\bar{B}(0,1)$ is
the unit sphere, a manifold with positive
curvature by standard convention. However,
$\bar B(0,1)$ can be expressed as $\{x: f(x) \geq 0 \}$
where $f(x)= 1- \|x\|^2$, or equivalently as
 $\{x: h(x) \leq 0 \}$ where $h(x)= \|x\|^2 - 1$.
Note that $Z(f)=Z(h)$,
but $-\nabla^2 f(x) \succ 0$ and
$+ \nabla^2 h(x) \succ 0$.

%
%{\bf  OLD \\
% and the sphere $Z$ as $f(x)=0$;
%note $-\nabla^2 f(x) \succ 0$.
%Also we can express $Z$ as $h(x)=0$ where $h=-f$
%in which case $+ \nabla^2 h(x) \succ 0$.

%end OLD \\}

However, on a nonsingular hypersurface $Z(f)$
one can designate its sides
by choosing one of the two normal directions  $\pm \nu(x)$
at points $x$ on $Z(f)$. We call one such determination
at some point, say $u$,  the {\it outward direction},
and then select, at each $x$, the continuous function $\nu(x)$ to
be consistent with this determination.
In the ball example, $\nabla f(x) = - 2x$ and we would
typically choose $\nu(x) = - \nabla f(x)$
to be the outward normal direction to $Z(f)$.
In the more general case described below equation
\eqref{newcoord},
let us call $-\nabla f(x)$ the outward normal,
which near the origin points away from the set
$\{ (x', y): y \geq \phi(x') \}$.
To see this, note that $ - \nabla f(0,0) = \left[%
\begin{array}{r}
 0 \\
 -1 \\
\end{array}%
\right]$.

\smallskip

We remark that the definition of positive curvature for some
hypersurface $Z$ at a nonsingular point is independent
of the choice of defining functions.
Suppose $f$ and $g$ are smooth defining functions such that
$Z\cap B(u,\dt)=Z(f)\cap B(u,\dt)=Z(g)\cap B(u,\dt),
\nabla f(u) \ne 0 \ne \nabla g(u)$ for some $\dt >0$ and
\be
\label{eq:Bf}
\{x\in B(u,\dt): f(x)\geq 0\}
\ = \ \{x\in B(u,\dt): g(x)\geq 0\} .
\ee
Then
the second fundamental form in terms of $f$ is positive definite
(resp. semidefinite) at $u$ if and only if
the second fundamental form in terms of $g$ is positive definite
(resp. semidefinite) at $u$.
To see this, note that $\nabla f(u) = \af \nabla g(u)$
for some scalar $\af \neq 0$,
because $\nabla f(u)$ and $\nabla g(u)$ are perpendicular
to the boundary of $Z$ at $u$.
Also $\af>0$ because of \reff{eq:Bf}.
Then in the new coordinate system $(y,x^\prm)$ defined in \reff{newcoord},
as we have seen earlier,
$Z$ has nonnegative (resp. positive) curvature at $u$ if and only if
the function $y = \phi(x^\prm)$ is convex (resp. strictly convex) at $u$,
which holds if and only if
either one of $f$ and $g$ has
positive definite (resp. semidefinite) second fundamental form.
So the second fundamental form of $f$ and $g$
are simultaneously positive definite or semidefinite.

\medskip

The smooth function $f(x)$ on $\re^n$
is said to be {\it strictly quasi-concave at $u$}
%if the level set $\{x\in\re^n:\, f(x)-f(u)=0\}$ of $f(x)$
%has positive curvature at $u$, i.e.,
if the condition~\reff{def:posCurv} holds.
When $\nabla f(u)$ vanishes, we require
$-\nabla^2 f(u) \succ 0$ in order
for $f(x)$ to be strictly quasi-concave at $u$.
For a subset $V\subset \re^n$, we say
$f(x)$ is {\it strictly quasi-concave on $V$} if
$f(x)$ is strictly quasi-concave on every point on $V$.
%all the level sets of $f(x)$ have positive curvature there.
%In words,
%a smooth function is strictly quasi-concave if and only if
%the corresponding level sets are all positively curved.
When $>$ is replaced by $\geq$ in \reff{def:posCurv},
we can similarly define $f(x)$ to be {\it quasi-concave}.
We remark that our definition of quasi-concavity here
is slightly less demanding than the usual
 definition of quasi-concavity
in the existing literature
(see Section~3.4.3 in \cite{BV}).
%for the definition
%of quasi-concave or quasi-convex functions
%by using first and second order conditions).

\medskip
%For the convenience of the reader,
Recall that a polynomial $g(x)$ is said to be
{\it sos-concave} if
$-\nabla^2 g(x) = W(x)^TW(x)$
for some possibly nonsquare matrix polynomial $W(x)$.
The following theorem gives sufficient conditions
for SDP representability
in terms of sos-concavity or strict quasi-concavity.

\begin{thm} (Theorem~2 \cite{HN1}) \label{thm:qscv}
Suppose $S=\{x\in\re^n:\, g_1(x)\geq 0, \cdots, g_m(x)\geq 0\}$ is
a compact convex set defined by polynomials $g_i(x)$
and has nonempty interior.
For each $i$, if $g_i(x)$ is either sos-concave or
strictly quasi-concave on $S$, then
$S$ is SDP representable.
\end{thm}

\subsection{Sufficient and necessary conditions on defining polynomials}

In this subsection, we give sufficient conditions
as well as necessary conditions
for SDP representability
for both basic and nonbasic convex semialgebraic sets.
These conditions are about the properties of
defining polynomials on the part of the boundary
where they vanish,
instead of the whole set.
This is different from the conditions given in \cite{HN1}.
Let us begin with a proposition which is often used later.

\begin{pro} \label{pro:cvxsdpcov}
Let $S$ be a compact convex set.
Then $S$ is SDP representable if and only if
for every $u\in \pt S$, there exists some $\dt>0$
such that $S \cap \bar B(u,\dt)$ is SDP representable.
\end{pro}
\begin{proof}
$``\Rightarrow"$\,
Suppose $S$ has SDP representation
{\small
\[
S = \left\{ x\in \re^n: A+\sum_{i=1}^n x_iB_i + \sum_{j=1}^N u_j C_j\succeq 0\right\}
\]
}
for symmetric matrices $A,B_i,C_j$.
Then $S\cap \bar B(u,\dt)$ also has SDP representation
{\small
\[
\left\{ x\in \re^n: A+\sum_{i=1}^n x_iB_i + \sum_{j=1}^N u_j C_j\succeq 0, \quad
\bbm I_n & x-u \\ (x-u)^T & \dt^2 \ebm \succeq 0 \right\}.
\]
}

$``\Leftarrow"$\,
Suppose for every $u\in \pt S$
the set $S \cap \bar B(u,\dt_u)$ has SDP representation for some $\dt_u >0$.
Note that $\left\{B(u,\dt_u):\,u\in \pt S\right\}$ is an open cover
for the compact set $\pt S$.
So there are a finite number of balls, say,
$B(u_1,\dt_1),\cdots, B(u_L,\dt_L)$, to cover $\pt S$.
Note that
\[
S = \cv{\pt S} =\cv{ \bigcup_{k=1}^L (\pt S \cap \bar B(u_k,\dt_k) ) \, }
\subseteq \cv{ \bigcup_{k=1}^L (S \cap \bar B(u_k,\dt_k) ) \, }
\subseteq  S.
\]
The sets $S\cap \bar B(u_k,\dt_k)$ are all bounded.
By Theorem~\ref{thm:mUsdp},
we know
\[ S=  \cv{ \bigcup_{k=1}^L S \cap \bar B(u_k,\dt_k) } \]
has SDP representation.
\end{proof}

\medskip

When the set $S$ is basic closed semialgebraic,
we have the following sufficient condition for SDP representability,
which strengthens Theorem~\ref{thm:qscv}.

\begin{theorem} \label{thm:bcsa}
Assume
$S =\{ x\in \re^n:\, g_1(x)\geq 0,\, \cdots,\, g_m(x)\geq 0 \}$
is a compact convex set defined by polynomials $g_i$
and has nonempty interior.
If for every $u\in \bdS$ and $i$ for which $g_i(u)=0$,
$g_i$ is either sos-concave or
strictly quasi-concave at $u$,
then $S$ is SDP representable.
\end{theorem}

\noindent
{\it Remarks:}
(i)
This result is stronger than Theorem~2 of \cite{HN1}
which requires each $g_i$ is either {\it sos-concave}
or strictly quasi-concave on the whole set $S$
instead of only on the boundary.
(ii)
The special case that some of the  $g_i$ are linear
is included in sos-concave case.
(iii) Later we
 will present a slightly weaker version of Theorem~\ref{thm:bcsa}
by using conditions on the curvature of the boundary
and give a very different but more geometric proof
based on Theorems~3 and 4 in \cite{HN1}.
This is left in \S \ref{sec:ghomi}.

\begin{proof}%%[Proof of Theorem~\ref{thm:bcsa}]
%By Proposition~\ref{prop:nicefunc},
%we can assume $g_i(x)$
%has negative definite Hessian
For any $u\in \bdS$,
let $I(u)=\{1\leq i\leq m: g_i(u)=0\}$.
For every $i\in I(u)$,
if $g_i(x)$ is not sos-concave,
$g_i(x)$ is strictly quasi-concave at $u$.
By continuity, there exist some $\dt >0$ such that
$g_i(x)$ is strictly quasi-concave on $\bar B(u,\dt)$.
%Then by Proposition~\ref{prop:nicefunc},
%there exists a polynomial $h_i(x)$ positive on $\bar B(u,\dt)$
%such that the product $p_i(x) = h_i(x) g_i(x)$
%has negative definite Hessian on $\bar B(u,\dt)$.
%When $g_i(x)$ is sos-concave, define $p_i(x)=g_i(x)$.
Note $g_i(u)>0$ for $i\notin I(u)$.
So we can choose $\dt >0$ small enough such that
\[
g_i(x) > 0,  \forall\, i \notin I(u), \,  \forall \, x\in \bar B(u,\dt).
\]
Therefore, the set $S_u :=S \cap \bar B(u, \dt)$ can be defined equivalently
by only using active $g_i$, namely,
\[
S_u = \left\{x\in\re^n:\, g_i(x) \geq 0,\forall\, i\in I(u), \, \dt^2 - \|x-u\|^2 \geq 0\right\}.
\]
For every $i \in I(u)$, the defining polynomial $g_i(x)$ is either sos-concave or
strictly quasi-concave on $S_u$.
%\[
%-\nabla^2 p_i(x) \succ 0, \forall \, x\in \bar B(u,\dt_u).
%\]
Obviously $S_u$ is a compact convex set with nonempty interior.
By Theorem~\ref{thm:qscv}, $S_u$ is SDP representable.
And hence by Proposition~\ref{pro:cvxsdpcov}, $S$ is also SDP representable.
\end{proof}

\medskip

Now we turn to the SDP representability problem
when $S$ is not basic  semialgebraic.
Assume $S=\bigcup_{k=1}^m T_k$ is compact convex.
Here each
$T_k = \{ x\in \re^n:\, g^k_1(x)\geq 0,\, \cdots,\, g^k_m(x)\geq 0 \}$
is basic closed semialgebraic but
not necessarily convex.
Similar sufficient conditions on $T_k$ for
the SDP representability of $S$
can be established.

\begin{theorem}[Sufficient conditions for SDP representability]
\label{thm:cvnbsa}
Suppose $S=\bigcup_{k=1}^m T_k$
is a compact convex semialgebraic set with each
\[
T_k=
\{ x\in \re^n:\, g^k_1(x)\geq 0,\, \cdots,\, g^k_{m_k}(x)\geq 0 \}
\]
being defined by polynomials $g^k_i(x)$.
If for every $u \in \bdS$, and each $g^k_i$ satisfying $g^k_i(u)=0$,
$T_k$ has interior near $u$ and
$g^k_i(x)$ is either sos-concave or strictly quasi-concave at $u$,
then $S$ is SDP representable.
\end{theorem}

%\noindent
%{\it Remark:} By definition of strict quasi-concavity,
%if the defining polynomial is not sos-concave,
%then the above sufficiency conditions
%require the  boundary corresponding to it
%has positive curvature.

\begin{proof}
%The proof is almost the same as for Theorem~\ref{thm:bcsa}.
By Proposition~\ref{pro:cvxsdpcov},
it suffices to show that for each $u\in \bdS$
there exists $\dt>0$ such that
the intersection $S\cap \bar B(u,\dt)$
is SDP representable.
For each $u\in \bdS$,
let $I_k(u)=\{1\leq i\leq m_k: g^k_i(u)=0\}$.
By assumption, for every $i\in I_k(u)$,
if $g^k_i(x)$ is not sos-concave,
$g^k_i$ is strictly quasi-concave at $u$.
By continuity,  $g^k_i$ is strictly quasi-concave
on $\bar B(u,\dt)$ for some $\dt >0$.
%By Proposition~\ref{prop:nicefunc},
%there exists a polynomial $h^k_i(x)$ positive on $ \bar B(u,\dt)$
%such that the product $p^k_i(x) = g^k_i(x) h^k_i(x)$
%has negative definite Hessian on $\bar B(u,\dt)$.
%We define $p^k_i(x)$ to be $g^k_i(x)$ if
%$g^k_i(x)$ is sos-concave.
Note $g^k_i(u)>0$ for all $i\notin I_k(u)$.
So $\dt>0$ can be chosen sufficiently small so that
\[
g^k_i(x) > 0,  \forall\, i \notin I(u), \,  \forall \, x\in \bar B(u,\dt).
\]
\iffalse
there exists some $\dt_u>0$ such that,
for all $i\in I_k(u)$
the zero set $Z(g^k_i)$ has positive curvature
on $T_k \cap \bar B(u,\dt_u)$
and $g^k_i(x)>0$ for all $i\notin I_k(u)$
on the ball $\bar B(u,\dt_u)$.
By Proposition~\ref{prop:nicefunc},
there exists a polynomial $h^k_i(x)$ positive on $T_k \cap \bar B(u,\dt_u)$
such that the product $p^k_i(x) = g^k_i(x) h^k_i(x)$
has negative definite Hessian on $Z(g^k_i)\cap \pt (T_k \cap \bar B(u,\dt_u) )$.
Here $\pt (K)$ denotes the boundary of a set $K$.
We define $p^k_i(x)$ to be $g^k_i(x)$ if
$g^k_i(x)$ is sos-concave.
\fi
Then we can see
\[
T_k \cap \bar B(u,\dt_u) =
\left\{x\in\re^n:\, g^k_i(x)\geq 0, \forall\, i\in I_k(u),\, \dt^2- \|x-u\|^2 \geq 0 \right\}.
\]
For every $i \in I_k(u)$, the defining polynomial $g^k_i(x)$ is either sos-concave or
strictly quasi-concave on $T_k \cap \bar B(u,\dt_u)$.
%\[
%-\nabla^2 p^k_i(x) \succ 0, \forall \, x\in \bar B(u,\dt_u).
%\]
Hence, the intersection $T_k \cap \bar B(u,\dt_u)$
is a compact convex set with nonempty interior.
By Theorem~\ref{thm:qscv}, $T_k \cap \bar B(u,\dt_u)$
is SDP representable.
Therefore, by Theorem~\ref{thm:mUsdp}, we know
\[
S \cap \bar B(u, \dt) =
\cv{ S \cap \bar B(u, \dt) } =
\mbox{conv} \Big( \bigcup_{k=1}^m  T_k \cap \bar B(u,\dt) \, \Big)  =
\mbox{conv} \Big( \bigcup_{k=1}^m  \,\big(T_k \cap \bar B(u,\dt)\big) \, \Big)
\]
is also SDP representable.
\end{proof}

\medskip

If the defining polynomials of a compact convex set $S$
are either sos-concave or strictly quasi-concave on the part of the boundary
of $S$ where they vanish,
Theorem~\ref{thm:cvnbsa} tell us $S$ is SDP representable.
If $S$ is the convex hull of the union of such convex sets,
Theorem~\ref{thm:mUsdp} tells us that
$S$ is also SDP representable.
We now assert that
this is not very far from the necessary conditions
for $S$ to be SDP representable.

We now need give a short review of smoothness of
the boundary of a set.
Let $S=\bigcup_{k=1}^m T_k$ and
$T_k= \{ x\in \re^n:\, g^k_1(x) \geq 0,\, \cdots,\, g^k_{m_k}(x) \geq 0\}$
with $\bdS$ and $\pt T_k$ denoting their topological boundaries.
For any $u\in \pt T_k(u)$,
the active constraint set
$I_k(u)=\{1\leq i\leq m_k:\, g^k_i(u)=0\}$ is nonempty.

We say $u$ is a {\it nonsingular point} on $\pt T_k$
if $|I_k(u)|=1$ and $\nabla g^k_i (u) \ne 0$ for $i \in I_k(u)$.
$u$ is called a {\it corner point} on $\pt T_k$ if $|I_k(u)|>1$,
and is nonsingular if $\nabla g^k_i (u) \ne 0$ for every $i \in I_k(u)$.
For $u \in \pt S$ and $i\in I_k(u) \ne \emptyset$,
we say the defining function $g^k_i$ is {\it irredundant at $u$
with respect to $\bdS$}
(or just {\it irredundant at $u$}
if the set $S$ is clear from the context)
if there exists a sequence of nonsingular points
$\{ u_N \} \subset Z(g^k_i)\cap \pt S$
such that $u_N \to u$;
otherwise, we say $g^k_i$ is {\it redundant} at $u$.
We say $g^k_i$ is {\it nonsingular} at $u$
if $\nabla g^k_i(u)\ne 0$.
Geometrically, when $g_i^k$ is nonsingular at $u \in \pt S$,
$g_i^k$ being redundant at $u$ means that the constraint $g_i^k(x)\geq 0$
could be removed without changing $S \cap B(u,\dt)$
for $\dt >0$ small enough.
A corner point $u\in \pt T_k$ is said to be {\it nondegenerate} if
$g^k_i$ is both irredundant and nonsingular at $u$
whenever $i\in I_k(u)\ne \emptyset$.

%------
\iffalse
%OLD? delete
%
%We
%say $u$ is a {\it smooth point?} on $\pt T_k$
%if $|I_k(u)|=1$ and $\nabla g^k_i (u) \ne 0$ for $i \in I_k(u)$.
%$u$ is called a {\it corner point?} on $\pt T_k$ if $|I_k(u)|>1$.
%\\ ? Motivat irred here or in the intro?\
%For $u \in \pt S$ and $i\in I_k(u) \ne \emptyset$,
%we say the defining function $g^k_i$ is {\it irredundant at $u$
%with respect to $\bdS$}
%(or just {\it irredundant at $u$}
%if the set $S$ is clear from the context)
%if there exists a sequence of smooth points?
%$\{ u_N \} \subset Z(g^k_i)\cap \pt S$
%such that $u_N \to u$;
%otherwise, we say $g^k_i$ is {\it redundant} at $u$.
%We say $g^k_i$ is {\it nonsingular} at $u$
%if $\nabla g^k_i(u)\ne 0$.
%Geometrically, when $g_i^k$ is nonsingular at $u \in \pt S$,
%$g_i^k$ being redundant at $u$ means that the constraint $g_i^k(x)\geq 0$
%could be removed without changing $S \cap B(u,\dt)$
%for $\dt >0$ small enough.
%A corner? point $u\in \pt T_k$ is said to be {\it nondegenerate?} if
%$g^k_i$ is both irredundant and nonsingular at $u$
%whenever $i\in I_k(u)\ne \emptyset$.
%
%END OLD
%
%-----------
%%
\fi

\medskip

The following gives necessary conditions for SDP representability.

\begin{theorem} (Necessary conditions for SDP representability) \label{thm:neccbsa}
If the convex set $S$ is SDP representable,
%and has SDP lift $\hat S$
%with nonempty interior,
then the following holds:
\bnum
\item [(a)] The interior $\overset{\circ}{S}$ of $S$ is a finite union of
basic open semialgebraic sets, i.e.,
\[
\overset{\circ}{S} =\bigcup_{k=1}^m T_k, \quad
T_k= \{ x\in \re^n:\, g^k_1(x) > 0,\, \cdots,\, g^k_{m_k}(x) > 0\}
\]
for some polynomials $g^k_i(x)$.

\item [(b)] The closure $\overline{S}$ of $S$ is a finite union of
basic closed semialgebraic sets, i.e.,
\[
\overline{S}=\bigcup_{k=1}^m T_k, \quad
T_k = \{ x\in \re^n:\, g^k_1(x) \geq 0,\, \cdots,\, g^k_{m_k}(x) \geq 0\}
\]
for some polynomials $g^k_i(x)$
(they might be different from those in (a) above).

\item [(c)] For each $u\in \pt \overline{S}$ and $i\in I_k(u)\ne\emptyset$,
if $g^k_i$ from (b) is irredundant
and nonsingular at $u$,  then
$g^k_i$ is quasi-concave at $u$.

\enum

\end{theorem}

\noindent
{\it Remarks:} (i) The proof of Theorem~\ref{thm:neccbsa}
only depends on the fact that
$S$ is a convex semialgebraic set
with nonempty interior,
and does not use its SDP representation.
(ii) The polynomials $g_i^k(x)$ in item (b) might be different from
the polynomials $g_i^k(x)$ in item (a).
We use the same notations for convenience.

\begin{proof}

(a) and (b) can be seen immediately from Theorem~2.7.2 in \cite{BCR}.

(c) Let $u\in\pt \overline{S} \cap \pt T_k$. Note that $\overline{S}$ is a convex set
and has the same boundary as $S$.

First, consider the case that $u$ is a smooth point.
Since $\overline{S}$ is convex,
$\pt\overline{S}$ has a supporting hyperplane $u+w^\perp=\{u+x:\, w^Tx=0\}$.
$\overline{S}$ lies on one side of $u+w^\perp$ and so does $T_k$,
since $T_k$ is contained in $\overline{S}$.
Since $u$ is a smooth point, $I_k(u) = \{i\}$ has cardinality one.
For some $\dt >0$ sufficiently small, we have
\[
T_k \cap B(u,\dt) = \{ x\in \re^n: g^k_i(x)\geq 0, \dt^2 - \|x-u\|^2 >  0 \}.
\]
Note $u+w^\perp$ is also a supporting hyperplane of $T_k$ passing through $u$.
So, the gradient $\nabla g^k_i(u)$ must be parallel to $w$,
i.e., $\nabla g^k_i(u) = \af^k_i w$ for some nonzero scalar
$\af^k_i\ne 0$.
Thus, for all $0\ne v \in w^\perp$ and $\eps >0$ small enough,
the point $u+ \frac{\eps}{\|v\|}v$ is not in the interior of $T_k \cap B(u,\dt)$,
which implies
\[
g^k_i(u + \frac{\eps}{\|v\|}v ) \leq 0, \,\, \forall\, 0 \ne v \in w^\perp = \nabla {g^k_i(u)}^\perp.
\]
By the second order Taylor expansion, we have
\[
- v^T \nabla^2 g^k_i(u) v \geq 0 ,\,\, \forall \, 0\ne v\in \nabla {g^k_i(u)}^\perp,
\]
that is, $g^k_i$ is quasi-concave at $u$.

Second, consider the case that $u\in \pt \overline{S}$ is a corner point.
By assumption that $g^k_i$ is irredundant
and nonsingular at $u$, there exists a sequence of smooth points
$\{ u_N \} \subset Z(g^k_i)\cap \pt \overline{S}$ such that $u_N \to u$
and $\nabla g^k_i(u)\ne 0$.

So $\nabla g^k_i(u_N)\ne 0$ for $N$ sufficiently large. From
the above, we know that
\[
- v^T \nabla^2 g^k_i(u_N) v \geq 0 ,\,\, \forall \, 0\ne v\in \nabla {g^k_i(u_N)}^\perp.
\]
Note that the subspace $\nabla {g^k_i(u_N)}^\perp$
equals the range space of the matrix $R(u_N)$ where
$$
R(v):=I_n - \frac{1}{\big(g^k_i(v)\big)^Tg^k_i(v)}
 g^k_i(v) \big(g^k_i(v)\big)^T.
$$
So the quasi-concavity of $g_i^k$ at $u_N$ is equivalent to
\[
- R(u_N)^T \nabla^2 g^k_i(u_N) R(u_N) \succeq 0.
\]
Since $\nabla g^k_i(u)\ne 0$, we have $R(u_N) \to R(u)$
%\[
%R(u_N) \to R(u):=I_n -
%\frac{1}{\big(g^k_i(u)\big)^Tg^k_i(u)} g^k_i(u) \big(g^k_i(u)\big)^T.
%\]
Therefore, letting $N\to \infty$, we get
\[
- R(u)^T \nabla^2 g^k_i(u) R(u) \succeq 0,
\]
which implies
\[
- v^T \nabla^2 g^k_i(u) v \geq 0 ,\,\, \forall \, 0\ne v\in \nabla {g^k_i(u)}^\perp,
\]
that is, $g^k_i$ is quasi-concave at $u$.
%the subspace
%$\nabla {g^k_i(u_N)}^\perp$ converges to
%$\nabla {g^k_i(u)}^\perp$, i.e.,
%$\nabla {g^k_i(u_N)} \to  \nabla {g^k_i(u)}$
%as $N$ goes to infinity.
%By continuity, we have
%\[
%- v^T \nabla^2 g^k_i(u) v \geq 0 ,\,\, \forall \, 0\ne v\in \nabla {g^k_i(u)}^\perp,
%\]
%that is, $g^k_i$ is quasi-concave at $u$.
\end{proof}

We point out that in (c) of Theorem~\ref{thm:neccbsa}
the condition that $g^k_i$ is irredundant can not be dropped.
For a counterexample, consider the set
\[
S = \left\{ x\in \re^2: g^1_1(x):=1-x_1^2-x_2^2\geq 0,
\, g^1_2(x):=(x_1-2)^2+x_2^2-1\geq 0 \right\}.
\]
Choose $u=(1,0)$ on the boundary.
Then $g^1_2$ is redundant at $u$.
As we can see, $g^1_2$ is not quasi-concave at $u$.

By comparing
Theorem~\ref{thm:cvnbsa} and Theorem~\ref{thm:neccbsa},
we can see the presented sufficient conditions and necessary
conditions are pretty close.
The main gaps are between
the defining polynomials being positive versus nonnegative curvature
and between the defining polynomials being
singular or not on the part of the boundary where they vanish.
A case which bypasses the gaps is when
some defining polynomials are sos-concave.

\bigskip

As is obvious, the set of defining polynomials for a semialgebraic set
is not unique, e.g., the set remains the same
if each defining polynomial is replaced by its cubic power.
However, as we can imagine, if we use some set of defining polynomials, we can prove
the SDP representability of the set,
but if we use some other set of defining of polynomials,
we might not be able to prove that.
A simple example is that the set
$\{x: g(x):=(1-\|x\|^2)^3 \geq 0\}$ is obviously SDP representable but
none of our earlier theorems using $g(x)$ only can show this set is SDP representable.
This is because, so far, we have discussed the
SDP representability only from the view of the defining polynomials,
instead of from the view of the geometric properties of the convex sets.
Sometimes, we are more interested in the conditions on the geometry of convex sets
which is independent of defining polynomials.
This leads us to the next subsection
%of discussing SDP representability
of giving conditions
on the geometric properties.

\subsection{Sufficient and necessary conditions on the geometry}

In this subsection, to address the SDP representability of convex semialgebraic sets,
we give sufficient conditions and necessary conditions on the geometry of the sets
instead of on their defining polynomials.

\bigskip

A subset $V\subset \re^n$ is a {\it variety}
if there exist polynomials $p_1(x),\cdots,p_m(x)$ such that
$V=\{x\in\re^n:\, p_1(x)=\cdots =p_m(x)=0\}$.
Given a variety $V$, define the ideal $I(V)$ as
\[
I(V) = \left\{ p\in \re[x]:\, p(u)=0 \text{ whenever } u\in V\right\}.
\]
Let the ideal $I(V)$ be generated by polynomials
$q_1,\cdots,q_k$.
A point $u\in V$ is said to be a {\it nonsingular point} if the matrix
$[\frac{\pt q_i}{\pt x_j}(u)]$ has full rank.
$V$ is said to be a {\it nonsingular variety} if every point of $V$
is a nonsingular point.
Note that if two varieties $V_1,V_2$ are both nonsingular
at a certain point $u$,
then their intersection variety $V_1\cap V_2$ might
be singular at $u$.
A set $Z\subset \re^n$ is said to be {\it Zariski open}
if its complement in $\re^n$ is a variety.
We refer to \cite{BCR,CLO} for more on algebraic varieties.

\begin{lemma} \label{lem:convAll}
Let $S\subset \re^n$ be a compact convex semialgebraic set
with nonempty interior. Then
\bnum

\item [(i)]  \label{it:C0}
The interior $\overset{\circ}{S}$ is the union of
basic open semialgebraic sets, i.e.,
\[
\overset{\circ}{S} =\bigcup_{k=1}^m T_k, \quad
T_k:= \{ x\in \re^n:\, g^k_1(x) > 0,\, \cdots,\, g^k_{m_k}(x) > 0\}
\]
where $g^k_i(x)$ are polynomials.
Each $T_k$ is bounded
and its closure has boundary $\pt T_k$.

\item [(ii)] \label{it:C1}
The Zariski closure $\cV^k$ of each $\pt T_k$ is
the union $\cV^k = \cV^k_1 \cup \cV^k_2 \cup \cdots \cup \cV^k_{L_k}$
of irreducible varieties of dimension $n-1$ such that
$\cV^k_i \cap \pt T_k \nsubseteq  (\cup_{j \ne i} \cV^k_i) \cap \pt T_k$.
We can write these as $\cV^k_i =\{x\in\re^n: f^k_i(x)=0\}$ for some irreducible polynomials
$f^k_i(x)$ such that the ideal $I(\cV^k_i)$ is generated by $f^k_i(x)$.
Furthermore, if every $\cV^k_i$ containing $u \in \pt T_k$ is nonsingular
at $u$ (i.e., $\nabla f^k_i(u)\ne 0$),
then
for $r>0$ sufficiently small we have
\be
\label{eq:locdef}
\overline{T_k} \cap \bar B(u,r)  =
\big( \bigcap_{1\leq i\leq L_k} \cV^{k+}_i \big) \cap \bar B(u,r), \qquad
\cV^{k+}_i := \{x\in\re^n: f^k_i(x)\geq 0\}.
\ee

\item [(iii)] \label{it:C2}
For $u \in \bdS \cap \pt T_k \cap \cV^k_i$,
we say $\cV^k_i$ is {\it irredundant} at $u$
%with respect to $\bdS$
if there exists a sequence
$\{u_N\}\subset \bdS$ converging to $u$ such that
$S\cap B(u_N,\eps_N) = \cV^{k+}_i \cap B(u_N,\eps_N)$
for some $\eps_N>0$.
If $\cV^k_i$ is nonsingular
and irredundant at $u$,
%with respect to $\bdS$,
then the curvature of $\cV_i$ at $u$ is nonnegative.

\item [(iv)] \label{it:C3}
The nonsingular points in $\cV^k_i$ form a Zariski open subset of $\cV^k_i$
and their complement has a lower dimension than $\cV^k_i$ does.

\enum

\end{lemma}

In the above lemma, the irreducible varieties $\cV^k_i$
are called the {\it intrinsic varieties} of $\bdS$,
and the corresponding polynomials $f^k_i$ are
called the {\it intrinsic polynomials} of $S$.
Note that every $\cV^k_i$ is a hypersurface.
The intrinsic $\cV^k_i$ is called {\it irredundant}
if it is irredundant at every $u\in \bdS \cap \pt T_k \cap \cV^k_i$.
$\cV^k_i$ is called {\it redundant}
at $u\in \bdS \cap \pt T_k \cap \cV^k_i$
if it is not irredundant at $u$.
The set $\mc{B}=\{ \cV^k_i:1\leq k\leq m, 1 \leq i\leq L_k\}$
of irreducible varieties in (ii) above is called
a {\it boundary sheet} of $S$.
We remark that the boundary sheet $\mc{B}$ of $S$ is not unique.

\begin{exm} \label{exm:bdsheet}
Consider the compact convex set
$S=\{x\in \bar B(0,1): x_2\geq x_1^2 \, \text{ or } \, x\in \re^2_+ \}$.
Define irreducible varieties $\cV^k_i$ as follows
{\small
\[
\cV^1_1=\left\{x\in \re^n: 1-\|x\|^2 =0 \right\}, \ \
\cV^1_2=\left\{x\in \re^n: x_2-x_1^2 =0 \right\}
\]
\[
\cV^2_1=\cV^1_1, \ \
\cV^2_2=\left\{x\in \re^n: x_2 =0 \right\}, \ \
\cV^2_3(a)=\left\{x\in \re^n: x_1-a x_2^2 =0 \right\} \, (0\leq a\leq 1).
\]
}
They are the intrinsic varieties of $\bdS$.
For any $0\leq a\leq 1$,
$
\mathcal{B}(a) = \left\{ \cV^1_1,\cV^1_2,\cV^2_1,\cV^2_2,\cV^2_3(a)\right\}
$
is a boundary sheet of $S$. It is not unique.
$\cV^1_1,\cV^1_2,\cV^2_1,\cV^2_2$ are all irredundant,
%%at some boundary points,
while $\cV^2_3(a)$ is redundant at the origin.
\end{exm}

\medskip

\begin{proof}[Proof of Lemma~\ref{lem:convAll}]
Note that $S$ is the closure of its interior.
%Pick  any point  $u \in \bcC$  and
%pick an open set $\cO$ interior to $\cC$.
%The interior of the cone $K$ of all intervals joining
%$\cO$ to $u$ must lie in the interior of $\cC$ and
%points interior to $K$ can approach its vertex $u$.
Pick any point  $u \in \bdS$ and
pick an interior point $o$ to $S$.
The interior points of the interval joining
$o$ to $u$ must lie in the interior of $S$ and
can approach its vertex $u$.

(i) This is the claim (a) of Theorem \ref{thm:neccbsa}.

\smallskip
(ii)
$T_k$ is a component of
$$ \check T_k:= \{ x: g^k(x) >0 \}$$
where $g^k:= g^k_1 g^k_2  \cdots g^k_{m_k}$
and is what \cite{HV} calls an algebraic interior.
In other words, any bounded basic open semialgebraic set
is an algebraic interior.
Lemma 2.1 of \cite{HV} now tells us
that a minimum degree defining  polynomial $\tg^k$
for $T_k$ is unique up to a multiplicative constant.
Also it says that any other defining polynomial $h$ for $T_k$
equals $p \tg^k$ for some polynomial $p$.
Thus $\tg^k(v) = 0$ and $\nabla \tg^k(v) = 0$
implies $\nabla h(v) =0$. So the singular points of $h$
on $\pt T_k$ contain the singular points of $\pt T_k$.
Lemma 2.1 of \cite{HV}
characterizes the boundary of algebraic interiors.
The third and fourth paragraphs in the proof of Lemma 2.1 of \cite{HV}
show that the Zariski closure of $\pt T_k$
is a union of irreducible varieties $\cV^k_i$ each
of dimension $n-1$ which satisfy all requirements of
(ii) except equation \reff{eq:locdef}.
Without loss of generality, the sign of $f^k_i$ can be chosen
such that $f^k_i(x)$ is nonnegative on $T_k$.
\iffalse
%So for all $r>0$ sufficiently small, it holds that
%\[
%\overline{T_k} \cap \bar B(u,r)
%\subseteq  \big( \bigcap_{1\leq i\leq L_k}
%\cV^{k+}_i \big) \cap \bar B(u,r).
%\]
\fi
When every $\cV^k_i$ is nonsingular at $u \in \pt T_k \cap \cV^k_i$,
there exists $r>0$ small enough such that
every $\cV^k_i$ is a smooth hypersurface
on $\bar B(u,r)$.
So on $\bar B(u,r)$, a point $v$ is on the boundary of $T_k$
if and only if all $f^k_i(v)\geq 0$ and at least one $f^k_i(v)=0$;
on the other hand,
 $v$ is in the interior of $T_k$
if and only if all $f^k_i(v)>0$.
Therefore  equation \reff{eq:locdef} holds.

\smallskip
(iii) This is implied by item (c) of Theorem~\ref{thm:neccbsa}.

\smallskip
(iv) The $\cV_i$ above are irreducible algebraic varieties.
Thus by Proposition~3.3.14 of \cite{BCR} the desired conclusions on
the nonsingular points follows.
\end{proof}

\medskip

In terms of intrinsic varieties,
our main result about SDP representability is

\begin{thm}  \label{thm:acsvar}
Let $S$ be a compact convex semialgebraic set with
nonempty interior, and
$\mc{B}=\{ \cV^k_i:1\leq k\leq m, 1 \leq i\leq L_k\}$
be a boundary sheet of $S$
as guaranteed by Lemma~\ref{lem:convAll}.
Assume every hypersurface $\cV^k_i$ in $\mc{B}$ is nonsingular on
$\cV^k_i \cap \bdS$,
and has positive curvature at $u \in \cV^k_i \cap \bdS$ whenever
$\cV^k_i$ is redundant at $u$.
Then $S$ is SDP representable if (resp. only if)
for each $u\in \bdS\cap \cV^k_i$ the hypersurface $\cV^k_i$
has positive (resp. nonnegative)
curvature at $u$.
\end{thm}

\begin{proof}
The necessary side is (iii) of Lemma~\ref{lem:convAll}.
Let us prove the sufficient side.
By Proposition~\ref{pro:cvxsdpcov},
it suffices to show that
for every $u\in \bdS$ there exists $\dt>0$
such that $S\cap \bar B(u,\dt)$ is SDP representable.
Let $\cV^k_i$ and $f^k_i$ be given by Lemma~\ref{lem:convAll}.
Fix an arbitrary point $u\in \bdS$
and let $I_k(u)=\{1\leq i\leq L_k: \, u\in \cV^k_i\}$.
By the assumption of nonsingularity of $\cV^k_i$ on $\cV^k_i\cap \bdS$
and equation~\reff{eq:locdef} in Lemma~\ref{lem:convAll},
there is some $\dt >0$ small enough such that
\[
S \cap \bar B(u,\dt) = \bigcup_{k=1}^m \overline{T_k} \cap \bar B(u,\dt)
\]
\[
\overline{T_k} \cap \bar B(u,\dt) = \left\{x\in\re^n:\, f^k_i(x) \geq 0,\,
\forall i\in I_k(u),\, \dt^2 - \|x-u\|^2 \geq 0\right\}.
\]
Note that $f^k_i$ are irreducible polynomials
and nonsingular (their gradients do not vanish) on $\cV^k_i\cap \bdS$.
So the positive curvature hypothesis implies that
each $f^k_i(x)$ is strictly quasi-concave on $\bar B(u,\dt)$
(we can choose $\dt >0$ small enough to make this true).
Obviously $T_k \cap \bar B(u,\dt)$ is a bounded set.
By Theorem~\ref{thm:qscv} and Theorem~\ref{thm:mUsdp},
we know $S\cap \bar B(u,\dt)$ is SDP representable.
\end{proof}

In terms of intrinsic polynomials,
the above theorem can be reformulated as

\begin{thm}  \label{thm:intrpoly}
Let $S$ be a compact convex semialgebraic set with
nonempty interior, and
$f^k_i\,(1\leq k\leq m, 1 \leq i\leq L_k)$
be intrinsic polynomials of $S$
as guaranteed by Lemma~\ref{lem:convAll}.
Assume every $f^k_i$ is nonsingular on $Z(f^k_i) \cap \bdS$,
and strict quasi-concave at $u\in Z(f^k_i) \cap \bdS$ whenever
$f^k_i$ is redundant at $u$.
Then $S$ is SDP representable if (resp. only if)
for each $u\in \bdS$ and $f^k_i$ satisfying $f^k_i(u)=0$
the intrinsic polynomial $f^k_i$
is strictly quasi-concave (resp. non-strictly quasi-concave )
at $u$.
\end{thm}

\medskip
\noindent
{\it Remarks:}
(i) In the above two theorems,
we assume intrinsic varieties (resp. intrinsic polynomials) are
positively curved (resp. strictly quasi-concave)
on the part of the boundary
where they are redundant.
This assumption is reasonable,
because redundant intrinsic varieties (resp. intrinsic polynomials)
are usually not unique
and there is a freedom of choosing them.
(ii) As mentioned in the introduction,
 under the nonsingularity assumption,
the gap between sufficient and necessary conditions
is the intrinsic varieties being
positively curved versus nonnegatively curved
or the intrinsic polynomials being
strictly quasi-concave versus nonstrictly quasi-concave.
A case bypassing the gap is the intrinsic polynomials
being sos-concave, as shown in Theorem~\ref{thm:cvnbsa}.
Thus, in Example~\ref{exm:bdsheet},
we know the compact set there is SDP representable.
(iii) In Theorems~\ref{thm:acsvar} and \ref{thm:intrpoly},
to prove the necessary conditions, we have only used the
convexity of $S$ and its nonempty interior,
instead of the SDP representability of $S$.
Thus the necessary conditions in Theorems~\ref{thm:acsvar} and \ref{thm:intrpoly}
are still true
when $S$ is a convex semialgebraic set with nonempty interior.

\bigskip
\section{Convex hulls of nonconvex semialgebraic sets} \label{sec:ncvx}
\setcounter{equation}{0}

In this section, we consider the problem
of finding the convex hull of a nonconvex semialgebraic set $T$.
%Without loss of generality,
%assume $T$ is closed,
%otherwise consider its closure.
%No matter whether  $S$ is convex or not,
The convex hull $\cv{T}$
must be convex and semialgebraic (Theorem~2.2.1 in \cite{BCR}).
By Theorem~2.7.2 in \cite{BCR},
the closure of $\cv{T}$ is a union of basic closed semialgebraic sets.
A fundamental problem
in convex geometry and semidefinite programming
is to find the SDP representation of $\cv{T}$.
This section will address this problem and
prove the sufficient conditions and necessary conditions
for the SDP representability of $\cv{T}$ summarized
in the Introduction.

\medskip

Let $T$ be a compact nonconvex set with boundary $\pt T$.
%%Denote by $\mbox{conv}(T)$ the convex hull of $T$.
Obviously $\cv{T}$ is the convex hull of the boundary $\pt T$.
Some part of $\pt T$ might be in the interior of $\mbox{conv}(T)$
and will not contribute to $\mbox{conv}(T)$.
So we are motivated to define the {\it convex boundary} $\pt_c T$ of $T$ as
\begin{align} \label{convbdry}
\pt_cT = \left\{u\in T:\, \ell^T u = \min_{x\in T} \ell^Tx \, \text{ for some } \, \,
\ell\in \re^n \, \text{ with } \, \|\ell\| =1 \right\} \subseteq \pt T.
\end{align}
Geometrically, $\pt_cT$ is the maximum subset of $\pt T$
contained in $\pt\cv{T}$,
and the convex hull of $\pt_cT$ is still $\cv{T}$.

\begin{pro} \label{prop:cvxbdy}
If $T$ is compact,
then $\cv{\pt_cT}=\cv{T}$
and $\pt_cT$ is also compact.
\end{pro}
\begin{proof}
Obviously $\cv{\pt_c(T)}\subseteq \cv{T}$.
We need to prove $\cv{\pt_c(T)}\supseteq \cv{T}$.
It suffices to show that if $u\notin \cv{\pt_cT}$
then $u\notin\cv{T}$.
For any $u\notin\cv{\pt_cT}$, by the Convex Set Separation Theorem,
there is a vector $\ell$ of unit length and a positive number $\dt >0$ such that
\[
\ell^T u <  \ell^Tx - \dt, \ \ \forall \, x \in \cv{\pt_cT}.
\]
Let $v\in T$ minimize $\ell^Tx$ over $T$,
which must exist due to the compactness of $T$.
Then $v\in \pt_cT$ and hence
\[
\ell^T u <  \ell^Tv - \dt = \min_{x\in T} \ell^T x - \dt.
\]
Therefore, $u\notin \cv{T}$.

Clearly $\pt_cT$ is bounded and closed by its
definition.
So $\pt_cT$ is compact.
\end{proof}

\medskip
\noindent
{\it Remark:} If $T$ is not compact, then Proposition~\ref{prop:cvxbdy}
might not be true. For instance, for set $T=\{x\in\re^2: \|x\|^2 \geq 1\}$,
the convex boundary $\pt_c T = \emptyset$, but $\cv{T}$ is the whole space.
When $T$ is not compact, even if $\cv{\pt T}= \cv{T}$, it is still possible that
$\cv{\pt_c T} \ne  \cv{T}$.
As a counterexample, consider the set
\[
W = \{(0,0)\} \cup \{x\in\re_+^2:\, x_1x_2 \geq 1\}.
\]
It can be verified that $\cv{W} = \cv{\pt W}$,
$\pt_c W = \{(0,0)\}$ and $\cv{\pt_c W} \ne \cv{W}$.

\bigskip

Note that every semialgebraic set is a finite union of
basic semialgebraic sets
(Proposition~2.1.8 in \cite{BCR}).
To find the convex hull of a semialgebraic set $T$,
by Theorem~\ref{thm:mUsdp},
it suffices to find the SDP representation of the convex hull of
each basic semialgebraic subset of $T$.

\begin{theorem}
Let $T_1,\cdots,T_m$ be bounded semialgebraic sets.
If each $\cv{T_k}$ is SDP representable, then
the convex hull of $\cup_{k=1}^m T_k$
is also SDP representable.
\end{theorem}
\begin{proof}
By Theorem~\ref{thm:mUsdp}, it suffices to prove that
\[
\cv{ \bigcup_{k=1}^m T_k  } =
\cv{ \bigcup_{k=1}^m  \cv{T_k} }.
\]
Obviously, the left hand side is contained in the right hand side.
We only prove the converse. For every $j= 1,\ldots, m$, we have
%it holds
\[
\cv{ T_j} \subseteq \cv{ \bigcup_{k=1}^m T_k  }.
\]
Now taking the union of left hand side for $j= 1,\ldots, m$, we get
\[
\bigcup_{j=1}^m \cv{ T_j} \subseteq \cv{ \bigcup_{k=1}^m T_k  }.
\]
Taking the convex hull of the above on both sides results in
\[
\cv{ \bigcup_{j=1}^m \cv{ T_j} } \subseteq \cv{ \bigcup_{k=1}^m T_k  },
\]
which implies the equality at the beginning of this proof.
\end{proof}

\begin{pro}
\label{pro:sdpcover}
Let $T$ be a compact semialgebraic set.
Then $\cv{T}$ is SDP representable if
for every $u\in \pt_cT$, there exists $\dt>0$
such that $\cv{T\cap \bar B(u,\dt) }$ is SDP representable.
\end{pro}
\begin{proof}
Suppose for every $u\in \pt_cT$
the set $\cv{T\cap \bar B(u,\dt_u)}$ has SDP representation for some $\dt_u >0$.
Note that $\{B(u,\dt_u):\,u\in \pt_cT\}$ is an open cover
of the compact set $\pt_cT$.
So there are a finite number of balls, say,
$B(u_1,\dt_1),\cdots, B(u_L,\dt_L)$, to cover $\pt_cT$.
Noting
\[
\cv{\pt_cT} \subseteq \cv{ \bigcup_{k=1}^L \pt_cT\cap \bar B(u_k,\dt_k) }
\subseteq \cv{ \bigcup_{k=1}^L \cv{T\cap \bar B(u_k,\dt_k)} } \subseteq  \cv{T},
\]
by Proposition~\ref{prop:cvxbdy}, we have
\[
\cv{T}=
\cv{ \bigcup_{k=1}^L \cv{T\cap \bar B(u_k,\dt_k)} }.
\]
The sets $\cv{T\cap \bar B(u_k,\dt_k)}$ are all bounded.
By Theorem~\ref{thm:mUsdp},
we know $\cv{T}$ is SDP representable.
\end{proof}

\noindent
{\it Remark}: By Proposition \ref{pro:sdpcover},
to find the SDP representation of
the convex hull of a compact set $T$,
it is sufficient to find the
SDP representations of convex hulls of the intersections of
$T$ and small balls near the convex boundary $\pt_cT$.
This gives the bridge between
the global and local SDP representations
of convex hulls.

\medskip

In the following two subsections,
we prove some sufficient conditions and necessary conditions
for the SDP representability of convex hulls.
They are essentially generalizations of
Section~\ref{sec:cvbsa} and the results in \cite{HN1}.

\subsection{Sos-concavity or quasi-concavity conditions}

In Section~\ref{sec:cvbsa},
we have proven some sufficient conditions
and necessary conditions
for the SDP representability of compact convex sets.
%%%$S=\{x\in\re^n:\, g_1(x)\geq 0, \cdots, g_m(x) \geq 0\}$
%is SDP representable if the defining polynomials
%are either sos-concave or strictly quasi-concave on the part of the boundary  where they vanish.
In this subsection, we prove similar conditions
for the convex hulls of nonconvex sets.
Throughout this subsection,
consider the semialgebraic sets which have nonempty interior
(then there are no equality defining polynomials).
We begin with basic semialgebraic sets, and then consider
more general semialgebraic sets.

\begin{theorem}
\label{thm:cvhbcsasdp}
Assume $T=\{x\in\re^n:\, f_1(x)\geq 0, \cdots, f_m(x) \geq 0\}$
is a compact set defined by polynomials $f_i(x)$
and has nonempty interior near $\pt_c T$, i.e.,
for every $u\in \pt_c T$ and $\dt >0$ small enough,
there exists $v \in B(u,\dt)$ such that
$f_i(v)>0$ for all $i=1,\ldots,m$.
If for each $u\in \pt_cT$ and $i$ for which $f_i(u)=0$,
$f_i(x)$ is either sos-concave or
strictly quasi-concave at $u$,
then $\cv{T}$ is SDP representable.
\end{theorem}

\begin{proof}
By Proposition~\ref{pro:sdpcover},
we only need prove for every $u\in \pt_c T$ the set
$\cv{T\cap \bar B(u,\dt)}$
is SDP representable
for some $\dt >0$.
%%The proof is very similar to the proof of Theorem~\ref{thm:bcsa}.
For an arbitrary $u\in \pt_cT$,
and let $I(u)=\{1\leq i\leq m: f_i(u)=0\}$.
For any $i\in I(u)$,
if $f_i(x)$ is not sos-concave,
$f_i$ is strictly quasi-concave at $u$.
By continuity,
$f_i$ is strictly quasi-concave
on $\bar B(u,\dt)$ for some $\dt >0$.
%Then, by Proposition~\ref{prop:nicefunc}
%there exists a polynomial $h_i(x)$ positive on $\bar B(u,\dt)$
%such that the product $p_i(x) = h_i(x) f_i(x)$
%has negative definite Hessian on $\bar B(u,\dt)$.
%If $f_i(x)$ is sos-concave, define $p_i(x) = f_i(x)$.
Note $f_i(u)>0$ for all $i\notin I(u)$.
Therefore, by continuity,
the number $\dt>0$ can be chosen small enough that
$f_i(x)>0$ for all $x\in \bar B(u,\dt)$ and $i\notin I(u)$.
Then we can see
\[
T_u :=T\cap \bar B(u,\dt) = \{ x\in\re^n:\, f_i(x) \geq 0, \,\forall\, i\in I(u), \quad
\dt^2 - \|x-u\|^2 \geq 0 \}.
\]
For every $i\in I(u)$, the polynomial $f_i(x)$
is either sos-concave or strictly quasi-concave
on $T_u$.
Clearly, $T_u$ is a compact convex set with nonempty interior.
%satisfies
%\[
%-\nabla^2 p_i(x) \succ 0, \, \forall
%x \in B(u,\dt).
%\]
By Theorem \ref{thm:qscv},
%(or equivalently Theorem~2 in \cite{HN1}),
we know $\cv{T_u}=T_u$ is SDP representable,
since $T_u$ is convex.
\end{proof}

\medskip

Now we consider nonbasic semialgebraic sets
and give similar sufficient conditions.

\begin{theorem}[Sufficient conditions for SDP representability of convex hulls] \label{thm:nonbcsa}
Assume $T =\bigcup_{k=1}^L T_k$ is a compact semialgebraic set with
\[
T_k = \{ x\in \re^n:\, f^k_{1}(x)\geq 0,\, \cdots,\, f^k_{m_k}(x)\geq 0 \}
\]
being defined by polynomials $f^k_i(x)$.
If for each $u\in \pt_cT$ and $f^k_i$ for which $f^k_i(u)=0$,
$T_k$ has interior near $u$ and $f^k_i$ is either sos-concave or
strictly quasi-concave at $u$,
%$Z(f^k_i) = \{ x\in \re^n:\, f^k_i(x) = 0 \}$
%has positive curvature at $u$ and $T_k$ has interior near $u$,
then $\cv{T}$ is SDP representable.
\end{theorem}
\begin{proof}
%The proof is very similar to that for Theorem~\ref{thm:cvnbsa}.
By Proposition~\ref{pro:sdpcover},
it suffices to prove for each $u\in\pt_cT$,
there exists $\dt>0$ such that
$\cv{T\cap \bar B(u,\dt) }$ is SDP representable.
Fix an arbitrary $u\in \pt_cT$,
and let $I_k(u)=\{1\leq i\leq m_k: f^k_i(u)=0\}$.
By assumption, if $i\in I_k(u)$ and $f^k_i(x)$ is not sos-concave,
$f^k_i$ is strictly quasi-concave at $u$.
Thus, by continuity, there exists $\dt >0$ so that
$f^k_i$ is strictly quasi-concave on $\bar B(u,\dt)$.
%By Proposition~\ref{prop:nicefunc},
%there exists a polynomial $h^k_i(x)$ positive on $ \bar B(u,\dt)$
%such that the product $p^k_i(x) = f^k_i(x) h^k_i(x)$
%has negative definite Hessian on $Z(f^k_i)\cap \bar B(u,\dt) $.
%We can choose $\dt >0$ small enough such that
%$p^k_i(x)$ has negative definite Hessian on $\bar B(u,\dt) $.
%We define $p^k_i(x)$ to be $f^k_i(x)$ if
%$f^k_i(x)$ is sos-concave.
Note that $f^k_i(u)>0$ for all $i\notin I_k(u)$.
So $\dt>0$ can be chosen small enough such that
$f^k_i(x)>0$ for all $x\in \bar B(u,\dt)$ and $i\notin I_k(u)$.
Then we can see that
\[
T_k \cap \bar B(u,\dt_u) =
\left\{x\in\re^n:\, f^k_i(x)\geq 0, \forall\, i\in I_k(u),\, \dt_u^2- \|x-u\|^2 \geq 0 \right\}
\]
is a compact convex set with nonempty interior.
And, for every $i\in I_k(u)$, $f^k_i(x)$ is either sos-concave
or strictly quasi-concave on $\bar B(u,\dt) $.
By Theorem~\ref{thm:qscv}, the set $T_k \cap \bar B(u,\dt_u)$ is SDP representable.
By Theorem~\ref{thm:mUsdp},
\[
\cv{T\cap \bar B(u,\dt)} =
\cv{ \bigcup_{k=1}^L T_k \cap \bar B(u,\dt)}
\]
is also SDP representable.
\end{proof}

As in  Theorem~\ref{thm:neccbsa}, we can get
similar necessary conditions on the defining polynomials
of the nonconvex sets.

\begin{theorem}[Necessary conditions for SDP representability of convex hulls] \label{thm:ncvxnecc}
Assume $T =\bigcup_{k=1}^L T_k$ is a compact semialgebraic set with
\[
T_k = \{ x\in \re^n:\, f^k_{1}(x)\geq 0,\, \cdots,\, f^k_{m_k}(x)\geq 0 \}
\]
being defined by polynomials $f^k_i(x)$,
and assume its convex hull $\cv{T}$ is SDP representable.
For each $u\in \pt_cT$ and $i\in I_k(u)\ne\emptyset$,
if $f^k_i$ is nonsingular and irredundant
at $u$ with respect to $\pt \cv{T}$, then
$f^k_i$ is quasi-concave at $u$.
\end{theorem}
\begin{proof}
Note that the convex hull $\cv{T}$ is compact and $T \subset \cv{T}$.
By Theorem~2.7.2 of \cite{BCR},
there exist basic closed semialgebraic sets
$T_{L+1},\ldots, T_{M}$ such that
\[
\cv{T} =\bigcup_{k=1}^{M} T_k.
\]
Every $T_k$ for $k=L+1,\ldots,M$ can also  be defined in the form
\[
T_k = \{ x\in \re^n:\, f^k_1(x) \geq 0,\, \cdots,\, f^k_{m_k}(x) \geq 0\}
\]
for certain polynomials $f^k_i(x)$.
The sets $T_1,\ldots,T_L$ are
basic closed semialgebraic
subsets of $\cv{T}$ and $\pt_cT \subseteq \pt \cv{T}$.
Consider $\cv{T}$ as the set $\overline{S}$ in Theorem~\ref{thm:neccbsa}.
Then the conclusion of this theorem is a direct application of
item (c) of Theorem~\ref{thm:neccbsa}.
\end{proof}

%
%{ \bf OLD delete ???

%As in  Theorem~\ref{thm:neccbsa}, we can get
%similar necessary conditions on the defining polynomials
%of the convex hulls.

%\begin{theorem}[Necessary conditions for SDP representability of convex hulls]
%Suppose the convex hull of the semialgebraic set $T$ is defined by polynomials $g^k_i$
%\[
%\overline{\cv{T}} =\bigcup_{k=1}^m T_k, \quad
%T_k = \{ x\in \re^n:\, g^k_1(x) \geq 0,\, \cdots,\, g^k_{m_k}(x) \geq 0\},
%\]
%and $\cv{T}$ is SDP representable.
%For each $u\in \pt_cT$ and $i\in I_k(u)\ne\emptyset$,
%if $g^k_i$ is nonsingular and irredundant
%at $u$ with respect to $\pt \cv{T}$, then
%$g^k_i$ is quasi-concave at $u$.
%\end{theorem}
%\begin{proof}
%Note that the set $\cv{T}$ is convex and SDP representable.
%The sets $T_k$ are
%basic closed semialgebraic
%subsets of $\overline{\cv{T}}$ and $\pt_cT \subseteq \pt \overline{\cv{T}}$.
%Consider $\cv{T}$ as the set $S$ in Theorem~\ref{thm:neccbsa}.
%Then the proof is a direct application of
%Theorem~\ref{thm:neccbsa}.
%\end{proof}

%end OLD }

\subsection{The PDLH condition}
\label{sec:pdlh}

In the previous subsection,
the nonconvex semialgebraic sets are assumed to have
nonempty interior near the convex boundary $\pt_cT$,
and so there can be no equality defining polynomials.
Now, in this subsection,
we consider the more general nonconvex semialgebraic sets
which might have empty interior
and equality defining polynomials.
Then the sufficient conditions
in the preceding subsection
do not hold anymore.
We need another kind of sufficient condition:
{\it the positive definite Lagrange Hessian (PDLH) condition}.
As in earlier sections, begin with basic semialgebraic sets.

\medskip

Assume $T$ is a compact basic semialgebraic set of the form
\[
T=\left\{x\in\re^n:\, f_1(x)=\cdots = f_{m_1}(x) =0, h_1(x) \geq 0, \cdots, h_{m_2}(x)\geq 0\right\}.
\]
Let $\pt T$ be the boundary of $T$.
For $u\in \pt T$, we say $T$ satisfies the
{\it positive definite Lagrange Hessian (PDLH) condition} at $u$
if there exists $\dt_u >0$ such that,
for every unit length vector $\ell\in \re^n$
and every $0<\dt \leq \dt_u$,
the first order optimality condition holds
at any global minimizer for the optimization problem
\be \label{cond:pdlh}
\baray{rl}
\underset{x\in \re^n}{\min} &   \ell^Tx  \\
s.t. &   f_1(x)=\cdots = f_{m_1}(x) =0 \\
 &  h_1(x) \geq 0, \cdots, h_{m_2}(x)\geq 0 \\
 & \dt^2 - \|x-u\|^2 \geq 0
\earay
\ee
and the Hessian of the associated Lagrange function
is positive definite on the ball $\bar B(u,\dt)$.
To be more precise, let $m=m_2+1$ and $h_m(x)=\dt^2 - \|x-u\|^2$.
The associated Lagrange function of \reff{cond:pdlh} is
\[
\mc{L}(x) = \ell^Tx - \sum_{i=1}^{m_1} \lmd_i f_i(x)
- \sum_{j=1}^m \mu_j h_j(x)
\]
where %%$\lmd_1, \cdots, \lmd_{m_1}$ and
$\mu_1\geq 0, \cdots, \mu_m\geq 0$.
Let $v$ be a global minimizer of problem~\reff{cond:pdlh}.
Then the PDLH condition requires
\[
\ell =  \sum_{i=1}^{m_1} \lmd_i \nabla f_i(v)
+ \sum_{j=1}^m \mu_j \nabla h_j(v)
\]
for some $\lmd_i$ and $\mu_j \geq 0$,
and the Hessian of the Lagrange function satisfies
\[
\nabla^2 \mc{L}(x) = - \sum_{i=1}^{m_1}
\lmd_i \nabla^2 f_i(x) - \sum_{j=1}^m \mu_j \nabla^2 h_j(x) \succ 0,\,
\forall\, x \in \bar B(u,\dt).
\]

\medskip
\noindent
{\it Remark:}
The defined PDLH condition here is stronger than
the PDLH condition defined in \cite{HN1}.
This is because
the PDLH condition in \cite{HN1}
is defined for convex sets described by concave functions.
However, in this paper,
the set $T$ here is nonconvex.
We need stronger assumptions.

\bigskip

The next theorem is an extension of Theorem~1.1 in \cite{HN1}
to give sufficient conditions assuring the SDP representability
of $\cv{T}$.

\begin{theorem} \label{thm:pdlh}
%Let $T$ be a compact basic semialgebraic set of the form
Let $T=\{x\in\re^n:\, f_1(x)=\cdots = f_{m_1}(x) =0, h_1(x) \geq 0,
\cdots, h_{m_2}(x)\geq 0\}$
be a compact set defined by polynomials.
If the PDLH condition holds at every $u\in \pt_cT$,
then $\cv{T}$ is SDP representable.
\end{theorem}
\begin{proof}
By Proposition \ref{pro:sdpcover},
we only need prove for every $u\in \pt_cT$,
there exists $\dt >0$ such that $\cv{T\cap \bar B(u,\dt)}$
is SDP representable.
Let $\dt = \dt_u >0$ be given by the PDLH condition and
define $T_u = T\cap \bar B(u,\dt)$.
We now prove $\cv{T_u}$ is SDP representable.

\medskip

First, we construct the lifted LMI for $T_u$.
Let $m=m_2+1$ and $h_m(x)=\dt^2-\|x-u\|^2$.
For integer $N$, define the monomial vector
\[
[x^N] = \bbm 1 & x_1 &\cdots
& x_n & x_1^2 & x_1x_2 &\cdots & x_n^N  \ebm^T.
\]
Define new polynomials
$h^\nu(x) = h_1^{\nu_1}(x) \cdots h_m^{\nu_m}(x)$,
where $\nu=(\nu_1,\cdots,\nu_m) \in \Z_+^m$.
Let $d_\nu = \lceil \deg(h_1^{\nu_1}\cdots h_r^{\nu_m})/2 \rceil$
and $d_k = \lceil \deg(f_k)/2 \rceil$.
For a fixed integer $N \geq d_\nu, d_k$, define
{\small
\[
M_{N - d_\nu}(h^\nu y) =
\int_{\re^n}  h^\nu(x) [x^{N - d_\nu}][x^{N - d_\nu}]^T\, d \mu(x)
= \sum_{0 \leq |\af| \leq 2N } A_\af^\nu y_\af
\]
\[
f_k^Ty =
\int_{\re^n}  f_k(x) \, d \mu(x)
= \sum_{0 \leq |\af| \leq 2d_k } f_\af^k y_\af.
\]
}
Here $\mu(\cdot)$ can be any nonnegative measure such that $\mu(\re^n)=1$,
$y_\af = \int_{\re^n} x^\af d \mu(x)$ are the moments,
$A_\af^\nu$ are symmetric matrices, and $f_\af^k$ are scalars such that
{\small
\begin{align*}
h^\nu(x) [x^{N - d_\nu}][x^{N - d_\nu}]^T\,  &=  \,
\sum_{0 \leq |\af| \leq 2N } A_\af^\nu x^\af \\
f_k(x) \,  &=  \,
\sum_{0 \leq |\af| \leq 2d_k } f_\af^k x^\af.
\end{align*}
}
If $\supp(\mu) \subseteq T$, then we have $y_0=1$ and
\begin{align*}
\left.\baray{rl}
\forall \, \nu \in \{0,1\}^m, &
M_{N-d_\nu}(h^\nu y) \succeq 0 \\
\forall \, 1\leq k\leq m, &
f_k^Ty =0
\earay\right\}.
\end{align*}
Let $e_i$  denote the standard $i$-th unit vector in $\re^n$.
If we set $y_0=1$ and $y_{e_i}=x_i$
in the above LMI, then it becomes the LMI
\begin{align} \label{LmiN}
\left. \baray{rl}
\forall \, \nu \in \{0,1\}^m, &
A_0^\nu+\underset{1\leq i\leq n}{\sum} A_{e_i}^\nu x_i +
\underset{1< |\af|\leq 2N}{\sum} A_\af^\nu y_\af \succeq 0 \\
\forall \, 1\leq k \leq m, &
f_0^k+\underset{1\leq i\leq n}{\sum} f_{e_i}^k x_i +
\underset{1< |\af|\leq 2d_k}{\sum} f_\af^k y_\af = 0
\earay \right\}.
\end{align}
Obviously, the projection of LMI~\reff{LmiN}
to $x$-space contains $\cv{T_u}$.

\medskip
Second, we prove that
every linear polynomial nonnegative on $T_u$
has an SOS representation
with uniform degree bound.
Given any $\ell \in \re^n$ with $\|\ell\|=1$,
let $\ell^*$ be the minimum value of
$\ell^Tx$ over $T_u$ and $v\in T_u$
be a global minimizer.
By the PDLH condition,
there exist Lagrange multipliers
$\lmd_1, \cdots, \lmd_{m_1}$ and
$\mu_1\geq 0, \cdots, \mu_m\geq 0$ such that
\[
\ell =  \sum_{i=1}^{m_1} \lmd_i \nabla f_i(v)
+ \sum_{j=1}^m \mu_j \nabla h_j(v)
\]
and the Hessian of the Lagrange function satisfies
\[
\nabla^2 \mc{L}(x) = - \sum_{i=1}^{m_1}
\lmd_i \nabla^2 f_i(x) - \sum_{j=1}^m \mu_j \nabla^2 h_j(x)
\succ 0,\, \forall \, x\in \bar B(u,\dt).
\]
%The PDLH condition requires $\nabla^2 \mc{L}(x)$
%to be positive definite on
%$\bar B(u,\dt)$.
Since the Lagrange multipliers
$\lmd_i$ and $\mu_j$ are continuous functions of $\ell$
on the unit sphere,
there must exist constants $M>\eps>0$ such that for all $x\in B(u,\dt)$
\[
 M I_n \succeq  \int_0^1 \int_0^t \nabla^2 \mc{L}(v+s(x-v)) ds \, dt \succeq \eps I_n.
\]
By Theorem~27 in \cite{HN1}, there exist SOS matrix polynomials $G_\nu(x)$
such that
\[
\int_0^1 \int_0^t \nabla^2 \mc{L}(v+s(x-v)) ds \, dt  =
\sum_{\nu\in\{0,1\}^m} h_1^{\nu_1}(x) \cdots h_m^{\nu_m}(x) G_\nu(x)
\]
and the degrees of summand polynomials are bounded by
\[
\deg(h_1^{\nu_1}(x) \cdots h_m^{\nu_m}(x) G_\nu(x)) \leq
\Omega(\frac{M}{\eps}).
\]
Here $\Omega(\cdot)$ is a function depending on $T_u$.
Let $ f_\ell(x) = \mc{L}(x)-\ell^*$.
Then $ f_\ell(v)= 0$ and $\nabla f_\ell(v)=0$.
By Taylor expansion, we have
\begin{align*}
f_\ell(x) &=   (x-v)^T
\left(\int_0^1 \int_0^t \nabla^2 \mc{L}(v+s(x-v)) ds \, dt \right) (x-v) \\
& = \sum_{\nu\in\{0,1\}^m} \phi_\nu(x) h_1^{\nu_1}(x) \cdots h_m^{\nu_m}(x)
\end{align*}
where $\phi_\nu(x) =  (x-v)^TG_\nu(x)(x-v)$ are SOS scalar polynomials.
Since $\mu_j\geq 0$, let
\[
\sig_\nu(x) = \phi_\nu(x) +
\bca \mu_j & \text{ if } \nu = e_j \\  0 & \text{ otherwise } \eca
\]
be new SOS polynomials. Then we have
\[
\ell^Tx-\ell^* = \sum_{k=1}^{m_1} \lmd_k f_k(x) +
\sum_{\nu\in\{0,1\}^m} \sig_\nu(x) h_1^{\nu_1}(x) \cdots h_m^{\nu_m}(x).
\]
There is a uniform bound $N$ independent of $\ell$ such that
\be \label{eq:umdgbd}
\deg(f_k(x)), \deg(h_1^{\nu_1}(x) \cdots h_m^{\nu_m}(x) \sig_\nu(x)) \leq 2N.
\ee

\medskip
Third, we will show that
\reff{LmiN} is an SDP representation for $\cv{T_u}$
when $N$ is given by \reff{eq:umdgbd}.
In the above, we have actually shown that
a property called {\it Schm\"{u}dgen's
Bounded Degree Nonnegative Representation (S-BDNR)}
(see Helton and Nie~\cite{HN1}) holds, i.e.,
every affine polynomials $\ell^Tx-\ell^*$ nonnegative on $T$
belongs to the preordering generated by the $f_i's$ and $h_j's$
with uniform degree bounds.
This implies a weaker property called
the {\it Schm\"{u}dgen's Bounded Degree Representation (S-BDR)}
(see Lasserre~\cite{Las06}) holds, i.e.,
almost every affine polynomials $\ell^Tx-\ell^*$ positive on $T$
belongs to the preordering generated by the $f_i's$ and $h_j's$
with uniform degree bounds.
So Theorem~2 in \cite{Las06} can be applied to show that
the LMI~\reff{LmiN} is a SDP representation of $\cv{T}$.
For the convenience of readers, we give the direct proof here.
Since the projection of \reff{LmiN} to $x$-space
contains $\cv{T_u}$, it is sufficient to prove the converse.
In pursuit of a contradiction, suppose there exists a vector $(\hat
x, \hat y)$ satisfying \reff{LmiN} such that $\hat x \notin \cv{T_u}$.
By the Hahn-Banach Separation Theorem, there must exist a
unit length vector $\ell$ such that
\be  \label{eq:minell}
\baray{rl}
\ell^T \hat x < \ell^* =
\min & \ell^Tx \\
s.t. &  f_1(x) = \cdots f_{m_1}(x) = 0 \\
&   h_1(x)\geq 0, \cdots, h_m(x) \geq 0.
\earay
\ee
Let $v$ be the minimizer of $\ell^T x$ on $T_u$; of course $v
\in \pt T_u$.
By the PDLH condition,
there exist Lagrange multipliers
$\lmd_1,\cdots,\lmd_{m_1}$ and
$\mu_1,\cdots,\mu_m\geq 0$ such that
\[
\ell = \sum_{i=1}^{m_1} \lmd_i \nabla f_i(v) +
\sum_{j=1}^{m} \mu_j \nabla h_i(v), \quad
\quad \mu_j h_j(v) =
0,\quad \forall \, j =1 ,\cdots, m.
\]
As we have proved earlier,  the identity
\[
\ell^Tx - \ell^* = \sum_{k=1}^{m_1} \lmd_k f_k(x) +
\sum_{\nu\in\{0,1\}^m}
\sig_\nu(x) h_1^{\nu_1}(x) \cdots h_m^{\nu_m}(x)
\]
holds
for some SOS polynomials $\sig_\nu(x)$ with uniform degree bound
\[
\deg(\sig_\nu(x) h_1^{\nu_1}(x) \cdots h_m^{\nu_m}(x)) \leq 2N.
\]
Thus we can write $\sig_\nu(x) = [x^{N-d_\nu}]^T W_\nu [x^{N-d_\nu}]$
for some symmetric positive semidefinite matrix $W_\nu\succeq 0$.
In the above identity, replace each monomial $x^\af$ with
$|\af|>1$ by $\hat y_\af$, then we get,
for $\hat y_0 = 1$ and every $\hat y_{e_i} = \hat x_i$, \ldots,
\[
\ell^T\hat x - \ell^* = \sum_{k=1}^{m_1} \lmd_k \,
\left( \sum_{0\leq |\af|\leq 2d_k} f_\af^k \hat y_\af \right) +
\sum_{\nu\in\{0,1\}^m}
Trace\left( W_\nu \cdot  \big(\sum_{0\leq |\af|\leq 2N} A_\af^i \hat y_\af\big) \right) \geq 0,
\]
which contradicts \reff{eq:minell}.
%This shows that for every unit length vector $\ell$ the optimization problem
%\reff{eq:minell} is equivalent to the SOS program
%\begin{align*}
%\baray{rl}
%\max  & \quad \gamma \\
%s.t. &  \quad \ell^Tx - \gamma   = \sum_{i=1}^{m_1} \lmd_i f_i(x) +
%\underset{\nu\in\{0,1\}^m}{\sum} \sig_\nu(x) h^\nu(x)\\
%&\quad \text{ $\sig_\nu(x)$ are SOS }, \quad
%  \deg(\sig_\nu h^\nu) \leq 2N .
%\earay
%\end{align*}
%This can be written equivalently as
%\begin{align} \label{eq:SOSopt}
%\baray{rl}
%\max  & \quad \gamma \\
%s.t. &  \quad \ell^Tx - \gamma   =
%\underset{\nu\in\{0,1\}^m}{\sum} g^\nu(x)[x^{N-d_\nu}]^T
%W_\nu [x^{N-d_\nu}]\\
%&\quad W_\nu \succeq 0, \quad  \nu\in\{0,1\}^m. \earay
%\end{align}
%Note the identity
%\begin{align*}
%g^\nu(x)[x^{N-d_\nu}][x^{N-d_\nu}]^T  &  = \sum_{0\leq |\af|\leq
%2N} A_\af^\nu x^\af
%\end{align*}
%and we reiterate that $\ell^*$ is the maximum value of
%problem~\reff{eq:SOSopt}.
%The dual of the above SOS program can be shown to be
%\begin{align*}
%\min  & \quad \ell^Tx \\
%s.t. & \quad (x,y) \text{ satisfying } \reff{LmiN}.
%\end{align*}
%See \cite{HN1} for more details about the derivation of the dual problem.
%The point $\hat x$ we
%are separating from $\cv{T_u}$ is the projection of $(\hat x, \hat y)$
%satisfying \reff{LmiN}, and produces the
%value  $ \ell^T \hat x$. Next, as we have seen, $\ell^*$ is
%feasible for the above SOS program. Thus by
%the weak duality relation, we have
%$\ell^* \leq \ell^T \hat x$ which  contradicts \reff{eq:minell}.
\end{proof}

\iffalse
%\medskip
%\noindent
%{\it Remark.}
%During the proof of Theorem~\ref{thm:pdlh}, we have actually shown that
%the {\it Schm\"{u}dgen's Bounded Degree Nonnegative Representation (S-BDNR) property}
%(see \cite{HN1}) holds, i.e.,
%every affine polynomials $\ell^Tx-\ell^*$ nonnegative on $T$
%belongs to the preordering generated by the $f_i's$ and $h_j's$
%with uniform degree bounds.
%This implies that the {\it Schm\"{u}dgen's Bounded Degree Representation (S-BDR) property}
%(see \cite{Las06}) holds, i.e.,
%almost every affine polynomials $\ell^Tx-\ell^*$ positive on $T$
%belongs to the preordering generated by the $f_i's$ and $h_j's$
%with uniform degree bounds.
%So it is possible to apply Theorem~2 in \cite{Las06}
%to show the LMI~\reff{LmiN} is one SDP representation of $\cv{T}$.
\fi

\begin{theorem} \label{thm:pdlhU}
Let $T=\bigcup_{k=1}^L T_k$ be a compact semialgebraic set where
\[
T_k=\{x\in\re^n:\, f_{k,1}(x)=\cdots = f_{k,m_{k,1}}(x) =0, h_{k,1}(x) \geq 0,
\cdots, h_{k,m_{k,2}}(x)\geq 0\}.
\]
If for each $T_k$, the PDLH condition
holds at every $u\in  \pt_cT \cap \pt T_k$,
then $\cv{T}$ is SDP representable.
\end{theorem}
\begin{proof}
By Proposition~\ref{pro:sdpcover},
it suffices to prove for each $u\in\pt_cT$,
$\cv{T\cap \bar B(u,\dt) }$
is SDP representable for some $\dt >0$.
Fix an arbitrary $u\in \pt_cT$.
Let $I(u)=\{1\leq k\leq L: u \in \pt T_k\}$.
Then, by assumption, the PDLH condition holds at $u$ for every $T_k$
with $k\in I(u)$, and thus the radius $\dt>0$
required in the PDLH condition
can be chosen uniformly for all $k\in I(u)$ since $I(u)$ is finite.
Hence we have
\[
\cv{T\cap \bar B(u,\dt) }  = \cv{ \bigcup_{k\in I(u)} T_k\cap \bar B(u,\dt)   }
= \cv{ \bigcup_{k\in I(u)} \cv{ T_k\cap \bar B(u,\dt) }  } .
\]
By the proof of Theorem~\ref{thm:pdlh}, the set $\cv{ T_k\cap \bar B(u,\dt) }$ is SDP representable.
Therefore, by Theorem~\ref{thm:mUsdp},
$\cv{T\cap \bar B(u,\dt) }$ is also SDP representable.
\end{proof}

\noindent
{\it Remark:} \,
It should be mentioned that the PDLH condition is a very strong condition.
It requires that, when every linear functional is minimized over the nonconvex set $T \cap \bar B(u,\dt)$,
the first order KKT condition holds and that
the Hessian of the Lagrangian is positive definite
at the minimizer.
This might restrict the applications of
Theorem~\ref{thm:pdlhU} in some cases.

\bigskip

\section{A more geometric proof of Theorem \ref{thm:bcsa}}
\setcounter{equation}{0}
\label{sec:ghomi}

For which set $S$ does there exist a set of defining polynomials for
which the Lasserre-Parrilo type moment relaxations
produce an SDP representation of $S$?
The major challenge is that while
$S$ may be presented to us by polynomials
for which the Lasserre-Parrilo type constructions fail,
there might exist another set of defining polynomials
for which such a construction succeeds.
This requires us to be able to find a  set of
defining polynomials such that the Lasserre-Parrilo type constructions work.

This section presents a very different approach to
proving a similar version of Theorem \ref{thm:bcsa},
since what we did there used the localization technique heavily.
We shall show here that the Lasserre-Parrilo type moment construction
gives an SDP representation
by using a certain set of defining polynomials.
The proof we shall give,
based on Theorems~3 and 4 of Helton and Nie \cite{HN1}
and on the proof of a proposition of Ghomi \cite{Ghomi} (on
smoothing boundaries of convex sets),
is also very geometrical.

For the convex set
$S =\{ x\in \re^n:\, g_1(x)\geq 0,\, \cdots,\, g_m(x)\geq 0 \}$,
define $S_i =\{ x\in \re^n:\, g_i(x)\geq 0 \}$
and $Z_i = \{x\in\re^n: g_i(x)=0\}$.
The zero set $Z_i$ is a hypersurface.
Suppose $Z_i$ does not intersect the interior of $S$.
Then $Z_i \cap S = Z_i \cap \bdrS$
and so is contained in  the boundary of $S$.

In addition to the definition of positive curvature,
we need a hypothesis about the shape of
%the boundary piece
$Z_i \cap \pt S$.
We say
%the boundary piece
$Z_i \cap \bdrS$
has {\it strictly convex shape with respect to $S$}
if there exists a relative open subset
$Y_i\subset Z_i$ containing $Z_i \cap \bdrS$
such that for every $p\in \overline{Y}_i$
the set $S \cup \overline{Y}_i$ lies in one side of
the tangent plane $T_p(Z_i)$ of $Z_i$ at $p$,
and does not touch $T_p(Z_i)$ except $p$, that is,
$T_p(Z_i) \cap (S \cup \overline{Y}_i) \subseteq \{p\}$.
The notion of strictly convex shape
follows the notion of {\it strictly convex hypersurface} introduced in Ghomi \cite{Ghomi}.

\begin{theorem}
\label{thm:bcsaweak}
Let $S =\{ x\in \re^n:\, g_1(x)\geq 0,\, \cdots,\, g_m(x)\geq 0 \}$
be a compact convex set defined by polynomials $g_i$
and assume $S$ has  nonempty interior.
Assume $g_i(x)>0$ whenever $x$ is in the interior of $S$,
$\nabla g_i(u)\ne 0$ whenever $u\in Z_i \cap S$,
and $Z_i \cap \bdrS$ has strictly convex shape with respect to $S$
when $g_i(x)$ is not sos-concave.
If for each $u\in \bdS$ and every $i$ such that $g_i(u)=0$
we have either $g_i$ is sos-concave or
$Z_i$ has positive curvature at $u$,
then $S$ is SDP representable.
Moreover, there is a certain set of defining
polynomials for $S$ for which the
Lasserre-Parrilo moment construction
(5.4) and (5.6) given in \cite{HN1} gives an SDP representation.
\end{theorem}

\subsection{Background from \cite{HN1}}

First we review some results of \cite{HN1}
with slight modification of notation used
in the original version.
For a smooth function $f(x)$, the set
$\{x\in\re^n:\, f(x) \geq 0\}$
is called {\it poscurv-convex} if it
is compact convex, and its boundary $\pt T$ equals
$Z(f)=\{x\in\re^n:\, f(x) = 0\}$ which
is smooth ($\nabla f(x)$ does not vanish on $\pt T$)
and positively curved at every point
$u \in Z(f)$.
When $f(x)$ is restricted to be a polynomial,
the set
$\{x\in\re^n:\, f(x) \geq 0\}$
is said to be {\it sos-convex}
if $f(x)$ is sos-concave.

\medskip

\begin{theorem}  (Theorem~3 \cite{HN1})
\label{thm:setSimple}
Given polynomials $g_i$,
suppose
$S =\{x\in \re^n:\, g_1(x)\geq 0, \cdots, g_m(x)\geq 0\}$
is compact convex and has nonempty interior.
%Suppose for each $i$ the set $\zerS_i$
%does not intersect the interior of $S$.
If each $S_i: =\{x\in \re^n:\, g_i(x)\geq 0 \}$
is either sos-convex or poscurv-convex,
then $S$ is SDP representable.
\end{theorem}

We now turn to more general cases.
Recall that $Z_i =\{x\in \re^n:\, g_i(x) = 0 \}$.
We say $S_i =\{x\in \re^n:\, g_i(x)\geq 0 \}$
%%with $g_i$ a polynomial
is {\it extendable poscurv-convex with respect to $S$}
if $g_i(x)>0$ whenever $x$ lies in the interior of $S_i$ and
there exists a
poscurv-convex set $T_i=\{x: f_i(x) \geq 0\} \supseteq S$
such that $\pt T_i\cap S = \bdS_i \cap S$.
In other words, $\zerS_i\cap \bdrS$ can be extended to
become the boundary of a  poscurv-convex set.
Note that the condition of extendable poscurv-convexity of $S_i$
requires $Z_i$ does not intersect the interior of $S$.

\begin{theorem} (Theorem~4 \cite{HN1})
 \label{thm:posCurv}
Given polynomials $g_i$,
suppose $S =\{x\in \re^n:\, g_1(x)\geq 0, \cdots, g_m(x)\geq 0\}$
is compact convex and has nonempty interior.
%Suppose for each $i$
%the set $\zerS_i $ does not intersect the interior of $S$.
If each $S_i$ is
either sos-convex or extendable poscurv-convex with respect to
$S$, then $S$ is SDP representable.
\end{theorem}

\medskip

We re-emphasize that the proofs of these theorems in \cite{HN1}
provide a new set of defining
polynomials for $S$ (possibly bigger than the original set)
for which the Lasserre-Parrilo type moment constructions
(5.4) and (5.6) given in \cite{HN1} also produce SDP representations of $S$.

%{ \bf NEW and OLD  should we delete ???

%To insure an easy interface with \cite{HN1} for those who
%are interested, we mention
%a distinction between the set up in \cite{HN1} and here.
%In \cite{HN1} a prevailing hypothesis was:
%{the zero set
%$Z_i = \{x\in\re^n: g_i(x)=0\} $ has positive curvature
%on $Z_i \cap S$}.
%If we define positive curvature to have a meaning
%in the case when $\nabla g_i(u)=0$,
%i.e, like what we have done for defining quasi-concavity in \S 3,
%the main effect of positive curvature hypothesis
%in \cite{HN1} is to exclude
%$Z_i$ from intersecting the interior of $S$.
%So here we have included this as
%an additional hypothesis.
%%In the current paper we have changed  it to:
%%{ the zero set
%%$Z_i = \{x\in\re^n: g_i(x)=0\} $
%%has positive curvature on $Z_i \cap \bdS$}.

%end NEW OLD }

Comparing Theorems~\ref{thm:posCurv} and \ref{thm:bcsaweak},
we can see that
Theorem~\ref{thm:posCurv} implies Theorem~\ref{thm:bcsaweak}
if we can show
{\it $S_i$ is extendable poscurv-convex with respect to $S$
provided $Z_i$ has positive curvature on $S$}.
The main task of this section is to
prove this point and what is new to  the proof
is mostly in the facts about convex sets which we now
turn to.

\subsection{Smoothing boundaries of convex sets}

We begin with some notations.
Let $T_p(M)$ denote  the tangent plane at $p$ to a smooth hypersurface $M$
without boundary. Sometimes we need
the tangent plane on a hypersurface $\oM$ with boundary,
but this will not be a problem for us, because
$\oM $ encountered in this section will be always contained
in another smooth hypersurface $\tM$ without boundary.
In this case,  we still use the notation
$T_p(\oM)$ rather than $T_p(\tM)$.
For a point $ p \in \bbR^n$
and a set $B \subset \bbR^n$, define the distance
\[
dist(p, B) = \inf \{ \|p-b\|_2 : \ \   \ b \in B \}.
\]

For convex set $S$,
the set $Z_i = \{x\in\re^n: g_i(x)=0\}$ is a hypersurface in $\re^n$
and is smooth in a relatively open subset containing
$Z_i \cap \bdrS = Z_i \cap S$ by the nonsingularity of $Z_i \cap \bdrS$.
Suppose $U\subset Z_i$ is relatively open and $Z_i \cap \bdrS \subset U$.
Let $\nu:\, \overline{U} \to\mathbb{S}^{n-1}$ be the Gauss map,
the map given by the unit outward normal.
We determine the outward normal direction as follows.
The smooth positively curved hypersurface $Z_i \cap \bdrS$
 has at each point $p$ a unique direction $\pm \nu(p)$
 perpendicular to its tangent plane.
 The convex set $S$ lies in one side of the tangent planes of $\partial S \cap Z_i$.
 We select the  $+ \nu(p)$ for
 $p\in \partial S \cap Z_i$ to be pointed away from $S$ and
 call this the outward direction.
 The outward direction is uniquely determined by the continuity of $\nu(p)$ on $\overline{U}$.
 Under this determination of outward normal direction,
 for any $p \in \overline{U}$, we say a set $G$ lies to the inside (resp. outside) of
 the tangent plane  $T_p(\overline{U})$ if $\langle q-p, \nu(p) \rangle \leq 0$
 (resp.  $\langle q-p, \nu(p) \rangle \geq 0$) for all $q\in G$.
 Here $\langle \cdot, \cdot \rangle$ denotes the standard
 inner product in Euclidean spaces.

%
%
%%For a convex set $S$, let $Z_i \cap \bdrS$ be a piece of the boundary of $S$.
%We say the boundary piece $Z_i \cap \bdrS$ has {\it strictly convex shape} with respect to $S$
%if there exists a relative open subset $U\subset Z_i$ containing $Z_i \cap \bdrS$
%such that for any $p\in \overline{U}$ the set $S$ lies to the inside of $T_p(\overline{U})$
%and $\big(T_p(\overline{U})\backslash\{p\}\big) \cap S = \emptyset$.

%%The next lemma produces the extendability property
%%of ``pieces" of the boundary of a convex set.

The next lemma insures the extendability property
of ``pieces" of the boundary of a convex set.

\begin{lemma}
\label{lem:Sextendable}
Suppose $S$ is convex compact. Fix an index $i$.
Assume $\nabla g_i(u)$ is nonzero for every $u \in \zerS_i \cap \bdrS$,
the curvature of $\zerS_i$ is positive at all $u$ there, and
$Z_i$ does not intersect the interior of $S$.
If $Z_i \cap \bdrS$ has {\it strictly convex shape} with respect to $S$,
then $S_i$ is extendable poscurv-convex with
respect to $S$, i.e., there exists a convex set $T$ such that
\bnum
\item [(i)] The boundary $\pt T$ is nonsingular
(so is smooth) and has positive
curvature everywhere.
\item [(ii)] $T$ is compact,  $S\subset T$ and
$\zerS_i \cap \bdrS = \pt T \cap \bdrS  = \pt T \cap S$.
\enum
\end{lemma}

\noindent
{\it Proof of Lemma \ref{lem:Sextendable}:}
The proof we shall give
is very similar to the proof of Proposition~3.3 in Ghomi \cite{Ghomi}.
We need construct a set $T$ satisfying the conclusions of Lemma~\ref{lem:Sextendable}.
%This is just as in the proof of Proposition~3.3~\cite{Ghomi}
But our construction of $T$ is slightly different from the one given in \cite{Ghomi}.
We proceed the proof by showing Claims~A,B,C,D and E.

\medskip
\noindent
{\bf Claim~A} \,
There exists a relatively open subset $U \subset Z_i$ satisfying
\bit
\item [(1)]\label{cmA:incl}
$\zerS_i \cap \bdrS \subset U$;

\item [(2)]\label{cmA:bdd}
the closure $\overline{U}$ is compact;

\item [(3)]\label{cmA:poscurv}
$U$ is smooth and
$\overline{U}$ has positive curvature everywhere;

\item [(4)]\label{cmA:eq}
$S \cap U = \bdS \cap U =\zerS_i \cap \bdrS$;

\item [(5)]\label{cmA:epty}
the relative boundary $\pt \overline{U} := \overline{U} \smallsetminus U$
satisfies $\pt \overline{U} \cap S = \emptyset$;

\item [(6)]\label{cmA:insid}
for any $p \in \oU$, the set $S \cup \overline{U}$ lies strictly to the inside of
$T_p{(\overline{U})}$, that is,
it lies to the inside of $T_p(\oU)$ and
$(S\cup \oU) \cap \big(T_p(\overline{U})\backslash \{p\} \big)=\emptyset$.

\eit

\begin{proof}
We show that the set $U= \{x\in Z_i: dist(x, Z_i \cap \bdrS)< \eps\}$
satisfies all the conditions of Claim~A when
$\eps>0$ is sufficiently small.
Items (1), (2) are obvious.
Since $\nabla g_i(x)$ does not vanish on $\bdrS \cap Z_i$,
it also does not vanish on in $U$ when $\eps>0$ is sufficiently small.
From the algebraic definition of positive curvature in \reff{def:posCurv},
we also know $\overline{U}$ has positive curvature when $\eps>0$ is small.
So item (3) is also true.

For item (4),  we know that (1) implies
\[
\zerS_i \cap \bdrS  \subset  \bdS \cap U \subset S \cap U.
\]
To prove they are all equal to each other,
it suffices to show $S \cap U \subset \zerS_i \cap \bdrS$.
For any $a\in S \cap U$, the point $a$ must belong to $Z_i \cap \pt S$,
because otherwise $Z_i$ intersects the interior of $S$,
which contradicts an assumption of Lemma~\ref{lem:Sextendable}.
So $S \cap U \subset \zerS_i \cap \bdrS$ and then (4) holds.

For item (5),  note that $\pt \overline{U}= \{x\in Z_i: dist(x, Z_i \cap \bdrS)= \eps\}$.
If $\pt \overline{U}$ intersects $S$,
then there exists $a \in \pt \overline{U} \cap S$ such that $a\in Z_i$
and $dist(a,Z_i \cap \bdrS) = \eps >0$.
Hence $a \notin \bdrS$ and $a$ must belong to the interior of $S$,
which contradicts an assumption of Lemma~\ref{lem:Sextendable}.
So (5) holds.

%{\bf NEW \\

%For item (6), it is clear that $S$ lies strictly to the inside of the tangent plane $T_p(\overline{U})$
%for every $p\in \overline{U}$,  by the condition that $Z_i\cap\bdrS$
%has strictly convex shape with respect to $S$.
%We need prove this is also true for $\overline{U}$.
%Clearly $\bdrS\cap Z_i \subset S$ lies strictly to the inside of $T_p(\overline{U})$.
%For every $q \in \pt (\bdrS\cap Z_i)$ (the relative boundary of $\bdrS\cap Z_i$),
%by the geometric definition of positive curvature of $Z_i$ (see Paragraph~2 in Section~3.1),
%there exists a small ball $B(q,\eta)$ such that $Z_i \cap B(q,\eta)$ is the graph of
%a strictly convex function in a certain coordinate system.
%From the determination of outward normal direction,
%we know $Z_i \cap B(q,\eta)$ lies strictly to the inside of its tangent planes.
%Therefore, when $\eps>0$ is sufficiently small,
%the set $\overline{U}$ also lies strictly to the inside of $T_p(\overline{U})$.

%end NEW }

Item (6) is just from the condition that
$Z_i \cap S$ has strictly convex shape with respect to $S$.
\end{proof}

Fix a relatively open set $U$ satisfying Claim~A.
For any small $t$, define
\[
U_t:=\{p_t:=p-t\nu(p) | p \in U \}.
\]
By continuity, its closure is
\[
 \overline{U}_t:=\{p_t:=p-t\nu(p) | p \in \overline{U} \}.
\]
Note that $U_0 = U$ and $\overline{U}_0 = \overline{U}$.
Let $\pt \overline{U}_t$ be the relative boundary
of $\overline{U}_t$,
that is, $\pt \overline{U}_t= \overline{U}_t \backslash U_t$.
Then for  $t$ small it holds that
\[
\pt \overline{U}_t:=\{p_t:=p-t\nu(p) | \ \  p \in \pt \overline{U} \}.
\]
Clearly, $\pt \overline{U} \cap S \subseteq \pt
\overline{U} \cap (S\cap \zerS_i) = \emptyset$
as $S\cap \zerS_i \subset U$, hence
\[
dist(\pt \overline{U}, S):=\min_{p\in \pt \overline{U}} dist(p,S) > 0
\]
as both $S$ and $\pt \overline{U}$ are compact.
By $\pt \overline{U} \cap S = \emptyset$ (condition \reff{cmA:epty} of Claim~A)
and continuity of $\pt \overline{U}_t$, we have
\be
\label{eq:UtSempty}
\pt \overline{U}_t \cap S = \emptyset
\quad\, \forall \, t\in (-r, r)\
\ee
for all $r>0$ small enough.
\medskip

Now we give some elementary geometric facts about $\overline{U}$ and $\overline{U}_t$.

\medskip
\noindent
{\bf Claim~B} \, For $r>0$ sufficiently small, we have
\bit
\item [(i)]  \label{cmB: pcurv}
$U_r$ is smooth and $\overline{U}_r$ has positive curvature everywhere;

\item [(ii)] \label{cmB:insid}
$\overline{U}_r$ globally lies to the inside of the tangent plane $T_{p_r}(\overline{U}_r)$
at any $p_r \in \overline{U}_r$;

%\item [(iii)]
%$\overline{U}_r \subset  \partial$ $conv(\overline{U}_r)$;

\item [(iii)] \label{cmB:normal}
$\nu(p_r) = \nu(p)$ for all $p \in \overline{U}$;

\item [(iv)] \label{cmB:dist}
for every $p \in \overline{U}$, $dist(p,\overline{U}_r) = dist(p,T_{p_r}(\overline{U}_r))=r$.

\eit

\begin{proof}
Items (i)-(ii) are the conclusions of paragraph~1
in the proof of Proposition 3.3 \cite{Ghomi}.
So we refer to \cite{Ghomi} for the proof.

(iii) This is a basic fact in differential geometry,
but we include a proof here since it is brief.
The hypersurface $Z_i$
has a relatively open smooth subset $\widetilde{U} \supset \oU$.
Similarly as before, we define
\[
\widetilde{U}_t:=\{p_t:=p-t\nu(p) | p \in \widetilde{U} \}.
\]
Fix an arbitrary point $p\in \oU \subset \widetilde{U}$.
Let $\{\phi(t): t \in \re \} \subset \widetilde{U}$ be an arbitrary smooth curve passing through $p$,
%and  locally smooth at $p$,
say, $\phi(0) = p$.
Since $\nu(p)$ is the normal to $\widetilde{U}$ at $p$,
we have $\langle \nu(p), \phi^\prm(0) \rangle = 0$.
Then $\{\phi(t)- r \nu(\phi(t)) : t \in \re \} \subset \widetilde{U}_r$ is
a smooth curve passing through $p_r$.
The unit length condition $\| \nu(\phi(t)) \|_2^2 = 1$ of normals implies
\[
\langle \nu(\phi(t)), \nabla_{\phi} \nu(\phi(t))  \phi^\prm(t) \rangle = 0,
\, \forall \, t.
\]
In particular,
$\langle \nu(\phi(0)), \nabla_{\phi} \nu(\phi(0))  \phi^\prm(0) \rangle = 0$.
Thus we have
\[
\left\langle \nu(p), \frac{d(\phi(t)- r \nu(\phi(t))) }{dt}\Big|_{t=0} \right\rangle =
\left\langle \nu(p), \phi^\prm(0)  \right\rangle - r
\langle \nu(\phi(0)), \nabla_{\phi}  \nu(\phi(0))  \phi^\prm(0)  \rangle = 0.
\]
So the curve $\{\phi(t)- r \nu(\phi(t)): t \in \re \}$ in $\widetilde{U}_r$ is also perpendicular to $\nu(p)$.
By uniqueness of unit normals of smooth hypersurfaces, we have $\nu(p_r) = \nu(p)$.

 (iv) For every $p \in \overline{U}$, (iii) says $\nu(p_r) = \nu(p)$.
 So the point $p$ lies to the outside of the tangent plane $T_{p_r}(\overline{U}_r)$.
 Since $p = p_r + r \nu(p_r)$ and $\nu(p_r)$ is perpendicular to $T_{p_r}(\overline{U}_r)$ at $p_r$,
 we have $r = dist(p,T_{p_r}(\overline{U}_r))$.
 From (ii), we know that
 $\overline{U}_r$ lies to the inside of the tangent plane $T_{p_r}(\overline{U}_r)$.
 So
 \[
 dist(p, \overline{U}_r) \geq dist(p,T_{p_r}(\overline{U}_r))=r.
 \]
Since $p_r= p - r \nu(p) \in U_r$, we obtain
$ dist(p, \overline{U}_r)  \leq r$.
Therefore, we have $dist(p, \overline{U}_r) = dist(p,T_{p_r}(\overline{U}_r))=r$.
\end{proof}

\medskip

\noindent
{\bf Claim~C} \,  For any $q \in \overline{U}_r$,
the set $S\smallsetminus (\cup_{0\leq t < r} U_t)$
globally lies to the inside
of $T_{q_r}({\hat U}_r)$ when $r$ is sufficiently small.
\begin{proof}
We prove this claim in three steps.

{\it Step~1}\,
From item (ii) of Claim~B we know the set $\overline{U}_s$  lies to the
inside of all the tangent planes of $\overline{U}_s$ when $s>0$ is small enough.
For every $q_t \in \overline{U}_t$, the tangent plane $T_{q_t}(\overline{U}_t)$
always lies to the inside of the  tangent plane $T_{q_s}(\overline{U}_s)$ when $0\leq s\leq t$ are both small.
This is because $q_s = q_t + (t-s) \nu(q_t)$,
since $\nu(q_s) = \nu(q_t)$ from item (iii) of Claim~B.
Hence for $\dt>0$ small enough, the set $\overline{U}_t$ lies to the
inside of all the tangent planes of $\overline{U}_s$ whenever $0\leq s\leq t\leq \dt$.

{\it Step~2} \,
Fix a $\dt>0$ sufficiently small as required in Step~1. Define the set
\[
W_{\dt} =  S \smallsetminus ( \cup_{0\leq t < \dt} U_t).
\]
For $\eta>0$ sufficiently small, it holds that
\be \label{eq:Wdt}
W_{\dt} =  S \smallsetminus ( \cup_{-\eta <  t < \dt} U_t).
\ee
This is because $U_t$ for $t\in (-\eta,0)$ lies outside of $S$,
due to item (6) of Claim~A and item (iii) of Claim~B.

Next, we show that the set  $U_{(-\eta,\dt)}:= \cup_{-\eta <  t < \dt} U_t$ is open.
For this purpose, define function
\[
\psi(p, t, z) := \bbm p - t \nu(p) - z \\ -g_i(p) \ebm, \quad \forall \,
(p,t,z) \in \re^n \times (-\eta, \dt) \times \re^n.
\]
Note that its partial Jacobian is
\[
\nabla_{(p,t) }\psi(p,t,z) =
\bbm  I_n - t \nabla_p \nu(p) &  -\nu(p) \\
 - \nabla g_i(p)^T & 0 \ebm.
\]
From the choice of outward normal direction, we know
$\nu(p) = - \frac{\nabla g_i(p)}{\| \nabla g_i(p)\|}$. So
\[
\det(\nabla_{(p,t) }\psi(p,t,z)) =  \| \nabla g_i(p)\|
\left( \nu(p)^T \left( I_n - t \nabla_p \nu(p) \right)^{-1}\nu(p) \right)
\det\left( I_n - t \nabla_p \nu(p) \right).
\]
Fix an arbitrary point $p_t= p - t\nu(p) \in U_{(-\eta,\dt)}$.
Then $\psi(p,t, p_t) = 0$ and $\|\nabla g_i(p)\| >0$
(since $U$ is smooth).
If $\eta$ and $\dt$ are sufficiently small,
it holds that $\det(\nabla_{(p,t) }\psi(p,t,p_t)) > 0$
and hence $\nabla_{(p,t) }\psi(p,t,p_t))$ is nonsingular.
By the Implicit Function Theorem,
there exist a small open neighborhood $\mc{O}_{p_t}$ of $p_t$ in $\re^n$
and a small open neighborhood $\mc{O}_{p,t}$ of $(p,t)$ in $\re^n \times (-\eta, \dt)$
such that $\psi(w,s,q)=0$ defines a smooth function
$(w,s) = \zeta(q)$ with domain $\mc{O}_{p_t}$ and range $\mc{O}_{p,t}$.
That is, for every $q \in \mc{O}_{p_t}$,
we can find a unique $(w,s)$ in $\mc{O}_{p,t}$ such that
$q = w - s \nu(w)$ and $g_i(w) = 0$.
If we choose the open neighborhoods $\mc{O}_{p_t}$ and $\mc{O}_{p,t}$ sufficiently small,
$w$ must be sufficiently close to $p$ enough so that $w\in U$
and $s\in (-\eta,\dt)$. So $q \in U_{(-\eta,\dt)}$.
This says $U_{(-\eta,\dt)}$ is an open set in $\re^n$.

Now we show that $W_{\dt}$ also lies to the inside of
the tangent planes of $U_r$ for all $r>0$ small enough,
by generalizing the argument in the proof in Proposition~3.3 in \cite{Ghomi}.
From the openness of $ \cup_{-\eta <  t < \dt} U_t$
and compactness of $S$,
we know $W_{\dt}$ is compact from \reff{eq:Wdt}.
For this purpose, define function $f_r: \overline{U}_0 \times W_{\dt} \to \re$ as
\[
f_r(p,a) = \langle a - p_r, \nu(p_r) \rangle, \quad
\forall \, (p, a) \in \overline{U}_0 \times W_{\dt},
\]
which is the signed distance between $a$ and $T_{p_r}(\overline{U}_0)$
(See \cite{Ghomi}).
By item (6) of Claim~A,
for every point $p \in \overline{U}_0 = \overline{U}$,
the convex set $S$ lies to the inside of the tangent plane $T_p(\overline{U}_0)$
and $S \cap \big(T_p(\overline{U}_0)\backslash \{p\}\big) = \emptyset$.
Since $W_{\dt} \subset S$ and $W_{\dt} \cap \overline{U}_0 = \emptyset$,
we know $W_{\dt}$ lies strictly to the inside of
the tangent plane $T_p(\overline{U}_0)$,
meaning that it does not touch $T_p(\overline{U}_0)$.
Thus $f_0<0$ on the compact set $\overline{U}_0 \times W_{\dt}$.
By continuity, we know $f_r<0$ on $\overline{U}_0 \times W_{\dt}$
for $r>0$ small enough. This means the set
$W_{\dt}$ lies strictly to the inside of all the
tangent planes of $\overline{U}_r$ for $0\leq r \leq \dt$ is sufficiently small.

{\it Step~3}\,
For $r\in [0,\dt]$ sufficiently small, one has
\[
S \smallsetminus ( \cup_{0\leq t < r} U_t)  \subset
W_{\dt} \cup   ( \cup_{r \leq t < \dt } U_t).
\]
From Step~1, we know $\cup_{r \leq t < \dt } U_t$
lies to the inside of all the tangent planes of $\overline{U}_r$.
From Step~2, we know $W_{\dt}$
lies to the inside of all the tangent planes of $\overline{U}_r$.
So we immediately conclude that
$S \smallsetminus ( \cup_{0\leq t < r} U_t)$
lies to the inside of all the tangent planes of $\overline{U}_r$.
\end{proof}

For $r>0$ small enough, define two new sets
\[
W = \mbox{conv}\Big( \overline{U}_r \cup  \overline{S \smallsetminus ( \cup_{0\leq t < r} U_t)} \Big),
\quad
K = {W} + \bar{B}(0,r).
\]

\medskip
\noindent
{\bf Claim~D}\,
 For $r>0$ small enough, the set $K$ is compact convex and
\[
\partial K \cap \bdrS = \zerS_i \cap \bdrS.
\]
\begin{proof}
Convexity and compactness are obvious. Note that
\be
\label{eq:Kdist}
\pt K= \{ b: dist(b, {W} ) = r \}.
\ee
First, we prove the inclusion $\zerS_i \cap \bdrS \subset \partial K \cap \bdrS$.
Suppose $p \in \zerS_i \cap \bdrS \subset \overline{U}$, then
\[
 dist(p, \overline{U}_r) \ \geq \ dist(p, {W} ),
\]
because $\overline{U}_r \subset {W}$.
From item (ii) of Claim~B we know the set $\overline{U}_r$
lies to the inside of the tangent plane $T_{p_r}(\overline{U}_r)$,
and from Claim~C we know $S \smallsetminus ( \cup_{0\leq t < r} U_t)$
lies to the inside of $T_{p_r}(\overline{U}_r)$.
Thus, by the definition of $W$,
the set $W$ also  lies to the inside of $T_{p_r}(\overline{U}_r)$.
Since $p$ lies to the outside of  $T_{p_r}(\overline{U}_r)$, we have
\[
  dist(p, \tanpr ) \leq \ dist(p, {W} ).
\]
%Thus $dist(p, \tanpr ) = \ dist(p, {W} )$.
Then from item (iv) of Claim~B we can see that
\[
 r = dist(p, \overline{U}_r) \ =  dist(p, \tanpr ) = \ dist(p, {W} ).
\]
So $dist(p,W)=r$ and hence $p \in \pt K \cap \bdS$ from \reff{eq:Kdist}.
Hence it holds $\zerS_i \cap \bdrS \subset \partial K \cap \bdrS$.

Second, we prove the reverse inclusion $ \partial K \cap \bdrS \subset \zerS_i \cap \bdrS$.
%Suppose $p \in \partial K \cap \bdrS$.
Start by noting that
\[
\bdrS =  (Z_i \cap \bdrS) \cup  \big(\bdrS \smallsetminus  (Z_i \cap \bdrS) \big).
 \]
We set about to prove $\bdrS \smallsetminus (Z_i \cap \bdrS)$ lies in the interior of $K$.
Consider $a \in \bdrS \smallsetminus (Z_i \cap \bdrS)$.
If $a \in  S \smallsetminus ( \cup_{0\leq t < r} U_t)$,
then $a\in W$ and hence $a+B(0,r/2)\subset K$ which implies $a$ is in the interior of $K$.
If $a \notin  S \smallsetminus ( \cup_{0\leq t < r} U_t)$,
then $a \in U_s$ for some $s \in (0,r)$
because $a\notin U_0$.
By definition of $U_s$ and $U_r$,
there exists $b\in U_r$ such that $a = b + (r-s) \nu(q)$
for some $q\in U_0$.
Since $b\in W$ and $\|a-b\| = r-s$,
we know $a + B(0, s/2) \subset b + B(0, r-s/2) \subset K$
and hence $a$ is also in the interior of $K$.
Combining the above,
we know $\bdrS \smallsetminus (Z_i \cap \bdrS)$ lies in the interior of $K$
and hence does not intersect $\pt K$.
Thus $\pt K \cap \pt S = \pt K \cap (Z_i \cap \pt S) \subset Z_i \cap \pt S$,
%it must holds $p \in \zerS_i \cap \bdrS$.
%So $\zerS_i \cap \bdrS \subset \partial K \cap \bdrS$,
which completes the proof.
\end{proof}

The proof from here on is essentially the same as in
Proposition~3.3 \cite{Ghomi}, so we could refer
to that but include here a slightly annotated version
for convenience.
The next step is to define a set $K^\eps$ which
is a small perturbation of $K$ and
which we shall prove has the properties
our lemma requires.
Let $V \subset U$ be an open set with
$\zerS_i \cap \bdrS \subset V \subset U$.
Set $U^\prm = \nu(U)$, and $V^\prm = \nu(V )$.
Then
$U^\prm$ and $V^\prm$ are open in $\mathbb{S}^{n-1}$, because
(since the second fundamental form of $U$ is nondegenerate)
$\nu$ is a local diffeomorphism.
Let $\bar \phi: \mathbb{S}^{n-1}\to \re$ be a
smooth function with support $\supp(\bar \phi) \subset U^\prm$,
and $\bar \phi |_{\overline{V}^\prm} \equiv 1$.
Let $\phi$ be the
extension of $\bar\phi$ to $\re^n$ given by $\phi(0) = 0$ and
$\phi (p) := \bar\phi(p/\|p\|)$, when $p\ne 0$.
Define $\bar h :\re^n \to \re$  by
\[
\bar h^\eps (p) := \tilde h^\eps(p) + \phi(p)( h(p) - \tilde h^\eps(p)),
\]
where $h$ is the support function of $K$, that is,
\[
h(p): = \sup_{x \in K} \langle p, x \rangle
\]
and $\tilde h^\eps$ is the {\it Schneider transform} of $h$
\[
\tilde h^\eps(p): = \int_{\re^n} \, h(p+\|p\| x ) \theta_\eps(
\|x\| ) dx.
\]
Note that $\tilde h^\eps$ is a convex function
(see Ghomi \cite{Ghomi}).
Here $\theta_\eps: \, [0, \infty) \to [0, \infty)$ is a smooth
function with $\supp(\theta_\eps) \subset [\eps/2, \eps]$ and
$\int_{\re^n} \theta_\eps( \|x\| ) dx = 1$. $\bar h^\eps$ supports
the convex set
\[
K^\eps := \{ x \in \re^n : \langle x, p\rangle \leq \bar
h^\eps(p),\, \forall p \in \re^n \}.
\]

\medskip
\noindent
{\bf Claim~E} The set $T= K^\eps$ satisfies the conclusions
of Lemma~\ref{lem:Sextendable}  when $\eps>0$ is
sufficiently small.
\begin{proof}
(i) We show $K^\eps$ is a convex body with support function
$\bar h^\eps$. To see this, it suffices to check that $\bar
h^\eps$ is positively homogeneous and convex. By definition, $\bar
h^\eps$ is obviously homogeneous. Thus to see convexity, it
suffices to show that $\nabla^2 \bar h^\eps(p)$ is nonnegative
semidefinite for all $p \in \mathbb{S}^{n-1}$. Since $\bar h^\eps
|_{\mathbb{S}^{n-1} \smallsetminus U^\prm} = \tilde h^\eps$, and $\tilde h^\eps$
is convex, we need to check this only for $p \in U^\prm$. To this
end, note that, for each $p \in U^\prm$ ,
$\nabla^2( h|_{T_p \mathbb{S}^{n-1}}) \succ 0$.
Here $T_p$ denotes the tangent plane at $p$.
Further, by construction,
\[
\|  h - \bar h^\eps \|_{ C^2(\overline{U}^\prm) } \to 0.
\]
So, for every $p \in \overline{U}^\prm$ , there exists an $\eps(p) > 0$
such that $\bar h^\eps  |_{T_p \mathbb{S}^{n-1}}$ has strictly
positive Hessian.
Since $\overline{U}^\prm$ is compact and $\eps(p)$ depends on
the size of the eigenvalues of the Hessian matrix of $\bar h^\eps
|_{T_p \mathbb{S}^{n-1}}$, which in turn depend continuously on
$p$, it follows that there is an $\eps > 0$ such that
$\nabla^2(\bar h^\eps |_{T_p \mathbb{S}^{n-1}}) \succ 0$
for all $p \in \overline{U}^\prm$.

(ii) We show that $\pt K^\eps$ is
nonsingular (hence smooth)
%ghomi states smooth only but look at proof
and positively curved.
By Lemma 3.1 in Ghomi \cite{Ghomi}, we only need check
$\nabla^2(\bar h^\eps  |_{T_p \mathbb{S}^{n-1}}) \succ 0$
for all $p\in \mathbb{S}^{n-1}$. For $p \in U^\prm$,
this was verified above.
For $ p \in \mathbb{S}^{n-1} \smallsetminus U^\prm$, note that $\bar h^\eps =
\tilde h^\eps$ on the cone spanned by  $\mathbb{S}^{n-1} \smallsetminus U^\prm$.
So it is enough to check that
$\nabla^2 (\tilde h^\eps  |_{T_p \mathbb{S}^{n-1}}) \succ 0$.
By Lemmas 3.2 and 3.1 of Ghomi \cite{Ghomi}, this
follows from the boundedness of the radii of curvature from below.

(iii)  Obviously $K^\eps$ is compact.
We show that $\zerS_i \cap \bdrS \subset \pt K^\eps$.
The proof is almost the same as the one of Proposition 3.3 in \cite{Ghomi}.
Since $\zerS_i \cap  \bdrS \subset U$, which
is smooth in $\pt K$, we have $h(p) = \langle \nu^{-1}(p), p
\rangle$, for all $p\in U^\prm$.
Apply the fact
%Lemma \ref{lem:grad} below to get
$\nabla h(p) = \nu^{-1}(p)$
to get
\[
\nu^{-1} (p) = \nabla h(p) = \nabla \bar h^\eps (p) = \bar
\nu^{-1}(p)
\]
for all $p\in V^\prm$,
where $\bar \nu$ is the Gauss map of $\pt K^\eps$
(see the proof of Proposition~3.3 in \cite{Ghomi}).
So $\zerS_i \cap \bdrS \subset
\bar \nu^{-1} (V^\prm )\subset \pt K^\eps$.

(iv) We show that $S \cap \pt K^\eps  = \bdrS \cap \pt K^\eps= Z_i \cap \bdS$.
Let $A:= \bar \nu^{-1} (V^\prm )$. Then $A\subset \pt K^\eps$,
as shown in (iii) above.
Since the Gauss map is continuous, $A$ is a relatively open subset of $V$.
Obviously $A \subset  U \subset \pt K$.
So the sets $\pt K^\eps\backslash A$, $\pt K\backslash A$
are all compact.
The set $S\backslash(\bdS\cap\zerS_i)$ is contained in the interior of $K$
(this has been proved in the proof of Claim~D),
so $S\cap \pt K = (\bdS\cap\zerS_i) \cap \pt K$.
From (iii) above, we know $\zerS_i \cap \bdrS \subset A$
and hence
$S \cap (\pt K\backslash A) = \emptyset$.
So it holds
\be \label{eq:SptKe}
\bdS \cap Z_i \subset A \subset \pt K^\eps.
\ee
Since $K^\eps \to K$ as $\eps \to 0$, it must hold that
$\pt K^\eps\backslash A \to \pt K \backslash A$
as $\eps \to 0$.
Thus, for $\eps >0$ small enough, we have
$S \cap (\pt K^\eps\backslash A)  = \emptyset$,
which implies (by using \reff{eq:SptKe})
\[
S \cap \pt K^\eps = (S \cap (\pt K^\eps \backslash A)) \cup (S \cap A) = S \cap A.
\]
Then we can see
\[
\bdS \cap Z_i %=  \bdS \cap \big(\bdS \cap Z_i\big)
\subset \bdS \cap \pt K^\eps \subset
S \cap \pt K^\eps = S \cap A  \subset S \cap U = \bdS \cap Z_i,
\]
where the last equality is by item (4) of Claim~A.
%From (iii) it is obviously that
%\[
%\bdS \cap Z_i \subset \bdS \cap \pt K^\eps \subset S \cap \pt K^\eps.
%\]
%Combining the above,
So all the intersections above are the same and hence we get
$S \cap \pt K^\eps  = \bdrS \cap \pt K^\eps= Z_i \cap \bdS$.

(v) We show that $S\subset K^\eps$.
Let $A$ be the relatively open subset of $V$ defined above.
Fix an interior point $v\in W \subset S$.
We proceed by contradiction.
If $S \not\subset K^\eps$, then the interior of $S$ is not contained
in the interior of $K^\eps$ since they are both compact.
So we can find an interior point $u \in S$ but $ u\notin K^\eps$.
Since $S$ and $K^\eps$ are convex, the line segment $L$ connecting $u$ and $v$
must be contained in $S$ and intersect $\pt K^\eps$, say, $b \in L \cap \pt K^\eps$.
Since $u,v$ are both in the interior of $S$,
$b$ must also be an interior point of $S$.
%Thus $b\not\in \bdS \cap Z_i$ (otherwise $\bdS \cap Z_i$
%contains an interior point of $S$ and then
%can not be a piece of boundary of $S$).
%Note, since $S$ is strictly convex,
%that $b$ must  be an interior point of $S$.
Thus $b\not\in \bdS \cap Z_i$.
We also must have $b \notin A$, because $S \cap A = \bdS \cap Z_i$.
So $b \in \pt K^\eps\backslash A$.
Since $b\in L \subset S$, we get
$b \in S \cap (\pt K^\eps\backslash A)$,
which is a contradiction since $S \cap (\pt K^\eps\backslash A) = \emptyset$,
as shown in (iv) above.
Therefore $S$ must be contained in $K^\eps$ for $\eps>0$ sufficiently small.
\end{proof}

Now that Claim~E is proved,
the proof of Lemma~\ref{lem:Sextendable} is finished.
\qed

\subsection{Proof of Theorem \ref{thm:bcsaweak}}

Given $u \in \bdrS$,
pick a $g_i$ for which $g_i(u)=0$.
By assumption, if $g_i$ is not sos-concave, then
each $\zerS_i$ has positive curvature at all $u$
in $\zerS_i \cap \bdrS $ and $\nabla g_i(u) \not =0$.
By Lemma \ref{lem:Sextendable},
$S_i$ is extendable poscurve-convex with respect to $S$.
Apply Theorem \ref{thm:posCurv}, noting that they
produce the desired Lasserre-Parrilo type moment
construction, to finish the proof.
\qed

\bigskip
\section{Conclusions}
\label{sec:concl}
For compact convex semialgebraic sets,
this paper proves the sufficient condition
for semidefinite representability:
each component of the boundary is nonsingular and has positive curvature,
and the necessary condition: the boundary components
have nonnegative curvature
when nonsingular.
We can see that the only gaps between them are
the boundary has singular points or
has zero curvature somewhere.
Compactness is required in the proof for the sufficient condition,
but it is not clear whether the compactness is necessary
in the general case.
So far, there is no
evidence that
SDP representable sets require more
than being convex and semialgebraic.
In fact, we conjecture that

\bigskip
\centerline{ \it Every convex semialgebraic set in $\re^n$
is semidefinite representable. }
\bigskip

The results of this paper are mostly on the
theoretical existence of semidefintie representations.
One important and interesting future work is to find concrete conditions
guaranteeing efficient and practical constructions of lifted LMIs
for convex semialgebraic sets and convex hulls
of nonconvex semialgebraic sets.

\bigskip
\bigskip
%\section{Acknowledgements}
\noindent
{\bf Acknowledgement}:
J. William Helton is partially supported by the
NSF through DMS 0700758, DMS 0757212 and the Ford Motor Company.
Jiawang Nie is partially supported by the
NSF through DMS 0757212.
The authors thank M. Schweighofer for numerous
helpful suggestions in improving the manuscript.
%The authors are grateful ?? to Raul Gomez for
% helping with the arguments in \S \ref{sec:ghomi}.

%\newpage
%\centerline{NOT FOR PUBLICATION}
%
%\tableofcontents

\end{document}